\documentclass{amsart}
\usepackage{amsthm}
\usepackage{amstext}
\usepackage{enumitem}
\usepackage{amssymb}
\usepackage{amsmath,calligra,mathrsfs}
\usepackage[all,cmtip]{xy}
\usepackage{dsfont}
\usepackage{hyperref}
\usepackage{graphicx}
\usepackage{colonequals}
\usepackage{graphbox}

\newcommand\bb{\mathbb}

\makeatletter
\renewcommand*\env@matrix[1][*\c@MaxMatrixCols c]{%
  \hskip -\arraycolsep
  \let\@ifnextchar\new@ifnextchar
  \array{#1}}
\makeatother

\theoremstyle{plain}
\newtheorem{thm}{Theorem}[section]
\newtheorem{lem}[thm]{Lemma}
\newtheorem{prop}[thm]{Proposition}
\newtheorem{cor}[thm]{Corollary}
\theoremstyle{defn}
\newtheorem{defn}[thm]{Definition}
\newtheorem{exmp}[thm]{Example}

\theoremstyle{remark}
\newtheorem{rem}{Remark}

\theoremstyle{plain}

\newtheorem{question}[thm]{Question}
\newtheorem{notation}[thm]{Notation}

\DeclareMathOperator{\Gal}{Gal}

\DeclareMathOperator{\sheafhom}{\mathscr{H}\text{\kern -3pt {\calligra\large om}}\,}

\DeclareMathOperator{\Br}{Br}

\DeclareMathOperator{\bPic}{{\bf Pic}}

\DeclareMathOperator{\Prym}{Prym}
\DeclareMathOperator{\PGL}{PGL}

\newcommand{\Deltatilde}{\tilde{\Delta}}
\newcommand{\kbar}{{\overline{k}}}
\newcommand{\Ptilde}{\tilde{P}}
\newcommand{\Ptildem}[1]{\tilde{P}^{(#1)}}
\newcommand{\Pm}[1]{{P}^{(#1)}}
\newcommand{\piDelta}{\varpi}

\newcommand{\piDoubleCover}{\tilde{\pi}}

\newcommand{\kk}{{\mathbf k}}
\newcommand{\Deltabar}{\Deltatilde^{-}}
\newcommand{\piDeltabar}{{\piDelta^{-}}}

\let\phi\varphi
\usepackage{comment}
\usepackage{xcolor,mathtools}

\title{Rationality of real conic bundles with quartic discriminant curve}

\author{Lena Ji}
\address{Department of Mathematics, University of Michigan, 530 Church Street, Ann Arbor, MI 48109-1043}
\email{lenaji.math@gmail.com}
\urladdr{http://www-personal.umich.edu/\~{}lenaji}

\author{Mattie Ji}
\address{Department of Mathematics, Brown University, Box 1917, 151 Thayer Street, Providence, RI 02912}
\email{mattie\_ji@brown.edu}

\begin{document}

\begin{abstract}
We study real double covers of $\mathbb P^1\times\mathbb P^2$ branched over a $(2,2)$-divisor, which are conic bundles with smooth quartic discriminant curve by the second projection. In each isotopy class of smooth plane quartics, we construct examples where the total space is rational. For five of the six isotopy classes we construct $\mathbb C$-rational examples with obstructions to rationality over $\mathbb R$, and for the sixth class, we show that the models we consider are all rational. Moreover, for three of the five classes with irrational members, we characterize rationality using the real locus and the intermediate Jacobian torsor obstruction of Hassett--Tschinkel and Benoist--Wittenberg. These double cover models were introduced by Frei, Sankar, Viray, Vogt, and the first author, who determined explicit descriptions for their intermediate Jacobian torsors.
\end{abstract}

\maketitle

\section{Introduction}

A fundamental question in algebraic geometry is the birational classification of algebraic varieties. The simplest varieties are those that are rational, i.e. birational to projective space. In this paper, we consider rationality over the field \(\mathbb R\) of real numbers, and when we write (stable/uni-)rationality without reference to the ground field, we will mean over \(\mathbb R\). We study the rationality of real conic bundle threefolds over \(\mathbb P^2\).

The discriminant double cover \(\Deltatilde\to\Delta\subset\mathbb P^2\) parametrizing the singular fibers of a conic bundle \(X\to\mathbb P^2\) is an important invariant that determines many of the properties of \(X\). If \(\deg\Delta\leq 3\) and \(X(\mathbb R)\neq\emptyset\), then results of Iskovskikh show \(X\) is rational, and if \(\deg\Delta\geq 6\) then \(X\) is irrational by work of Beauville (see Section~\ref{sec:ConicBundlesBg}). In this paper, we consider the case when \(\deg\Delta=4\). In this setting, \(X\) is \(\mathbb C\)-rational \cite{Iskovskikh-rationality-cbs}, but in general the geometric rationality construction need not descend to \(\mathbb R\), even if \(X(\mathbb R)\neq\emptyset\). It is natural to ask about the relationship between the real properties of \(\Deltatilde\to\Delta\) and the rationality of \(X\).
The classification of real smooth plane quartics \(\Delta\) has been the subject of classical interest:
Klein showed that in the moduli space of real plane quartics, the complement of the locus of singular quartics has six connected components, each corresponding to a real isotopy class \cite{Klein1876}.
These six real isotopy classes had earlier been classified by Zeuthen---empty, one oval, two nested ovals, two non-nested ovals, three ovals, and four ovals---who showed that curves in these classes have \(4\), \(4\), \(4\), \(8\), \(16\), and \(28\) real bitangents, respectively \cite{Zeuthen1874}.
We exhibit the following rationality behavior of the threefold \(X\) for different real isotopy classes of the quartic \(\Delta\):
\begin{thm}\label{thm:main}
Let \((**/*)\) denote the set of geometrically standard conic bundles \(X\to\mathbb P^2\) over \(\mathbb R\) with smooth quartic discriminant curve \(\Delta\) of topological type \(*\) and discriminant cover \(\Deltatilde\) of topological type \(**\).
\begin{enumerate}
	\item\label{item:0_oval} \((\emptyset/\emptyset)\) contains rational members, irrational members with points, and pointless members;
	\item\label{item:1_oval} \((\emptyset/\text{1 oval})\) contains both rational members and irrational members;
	\item\label{item:2_nonnested} \((\emptyset/\text{2 non-nested ovals})\) contains rational members, irrational members with connected real loci, and irrational members with disconnected real loci;
	\item\label{item:disconnected-types} \((\emptyset/\text{2 nested ovals})\) and \((\emptyset/\text{3 ovals})\) each contain both rational members and irrational members with disconnected real loci;
	\item\label{item:4_ovals} \((\emptyset/\text{4 ovals})\) contains rational members; and
	\item\label{item:Deltatildepoint_rational} If \(\Deltatilde(\mathbb R)\neq\emptyset\), then every member of \((**/*)\) is rational.
\end{enumerate}
\end{thm}

The irrational example in Theorem~\ref{thm:main}\eqref{item:1_oval} and the disconnected example in \eqref{item:2_nonnested} were constructed in \cite[Theorem 1.3]{FJSVV}; our contribution in these two cases is the construction of rational examples and an irrational connected example in \eqref{item:2_nonnested}. Theorem~\ref{thm:main}\eqref{item:Deltatildepoint_rational} is \cite[Proposition 6.1]{FJSVV}. 
All members with real points are unirational by \cite[Proposition 6.1]{FJSVV}, and the disconnected ones are not stably rational.

Conic bundles \(X\to\mathbb P^2\) with quartic discriminant curve were previously studied by the first author, together with S. Frei, S. Sankar, B. Viray, and I. Vogt in \cite{FJSVV}. When \(\deg\Delta = 4\), then \(X\) is \(\mathbb C\)-rational, and many obstructions to rationality vanish over \(\mathbb R\) \cite[Section 1.1]{FJSVV}; however, they constructed examples of irrational such \(X\) where the geometric rationality construction does not descend to \(\mathbb R\) (irrational examples also appeared implicitly in earlier work of Hassett--Tschinkel on real complete intersections of quadrics \cite[Remark 13 and Section 11.6]{ht-intersection-quadrics}). To study rationality in the degree \(4\) case, Frei--Ji--Sankar--Viray--Vogt introduced a particular model of these conic bundles that admits the structure of a double cover of \(\mathbb P^1\times\mathbb P^2\) branched over a divisor of bidegree \((2,2)\):
\begin{equation}\label{eqn:DoubleCoverEquation}
z^2=t_0^2 Q_1(u,v,w) + 2t_0t_1 Q_2(u,v,w) + t_1^2Q_3(u,v,w)\end{equation}
where \(Q_i\in\mathbb R[u,v,w]\) are quadratic forms. These double cover models admit the additional structure of a quadric surface bundle via the first projection. Using work of Bruin \cite{bruin} on \'etale double covers \(\Deltatilde\to\Delta\) of smooth plane quartics, \cite{FJSVV} showed that any such \(\Deltatilde\to\Delta\) can be realized as the discriminant cover of a conic bundle defined by an equation of the form \eqref{eqn:DoubleCoverEquation}. The Artin--Mumford sequence implies that up to a constant Brauer class, the discriminant double cover determines the birational isomorphism class of a conic bundle \cite[Section 6.9.6]{Poonen-Q-points}; thus, up to a class in \(\Br\mathbb R \cong\mathbb Z/2\), any conic bundle over \(\mathbb P^2\) with smooth quartic discriminant curve is birational over \(\mathbb P^2\) to such a double cover of \(\mathbb P^1\times\mathbb P^2\).

\cite{FJSVV} use the model~\eqref{eqn:DoubleCoverEquation} to construct examples of irrational conic bundles \(Y\) whose discriminant curves \(\Delta\) have real isotopy class one oval or two non-nested ovals. In each of these cases, irrationality is witnessed by a different obstruction. In the one oval example \cite[Theorem 1.3(2)]{FJSVV} \(Y(\mathbb R)\neq\emptyset\) is connected, and irrationality is shown using the \textsf{intermediate Jacobian torsor (IJT) obstruction}. This obstruction to rationality is a refinement over non-closed fields of the intermediate Jacobian obstruction of Clemens--Griffiths \cite{Clemens-Griffiths-ij}, and was recently introduced by Hassett--Tschinkel \cite{HT-cycle, ht-intersection-quadrics} and Benoist--Wittenberg \cite{bw-ij} (see Section~\ref{sec:IJT-background}).
However, the two non-nested ovals example \cite[Theorem 1.3(1)]{FJSVV} has no IJT obstruction to rationality but \(Y(\mathbb R)\) is disconnected; hence \(Y\) is irrational.
Thus, \cite{FJSVV} show that in general, neither the IJT obstruction nor the topological obstruction to rationality alone is sufficient to characterize rationality for conic bundle threefolds of the form \eqref{eqn:DoubleCoverEquation}.

For four of the six isotopy classes of the discriminant curve \(\Delta\), we prove the following characterizations of rationality for the double covers \eqref{eqn:DoubleCoverEquation}:

\begin{thm}\label{thm:DoubleCoverRationalityCriteria}
Over \(\mathbb R\), let \(Y\to\mathbb P^1\times\mathbb P^2\) be a double cover branched over a bidegree \((2,2)\) divisor, and assume that the discriminant cover \(\Deltatilde\to\Delta\) of the conic bundle obtained from the second projection is an \'etale double cover of a smooth quartic.
\begin{enumerate}
\item\label{item:0ovals-criterion}
If \(\Delta(\mathbb R)=\emptyset\), then \(Y\) is rational if and only if \(Y(\mathbb R)\neq\emptyset\) and the IJT obstruction vanishes.
\item\label{item:2nonnested-criterion}
If \(\Delta\) is two non-nested ovals, then \(Y\) is rational if and only if \(Y(\mathbb R)\) is connected and the IJT obstruction vanishes. Neither condition alone is sufficient to guarantee rationality.
\item\label{item:3ovals-rationality-criterion}
If \(\Delta\) is three ovals, then \(Y\) is rational if and only if \(Y(\mathbb R)\) is connected.
\item\label{item:4ovals-rational}
If \(\Delta\) is four ovals, then \(Y\) is rational.
\end{enumerate}
In each of these cases, \(Y\) is rational if and only if the quadric surface bundle \(Y\to\mathbb P^1\) admits a section.
\end{thm}

We do not know if the IJT obstruction is sufficient to characterize rationality in the one oval and two nested ovals cases. For two nested ovals case, the vanishing of the IJT obstruction implies \(Y(\mathbb R)\) is connected (Corollary~\ref{cor:two-nested-IJT}). In the case of one oval, the topological obstruction vanishes (Corollary~\ref{cor:components-of-Y-and-Delta}), and Frei--Ji--Sankar--Viray--Vogt have constructed an example where the IJT obstruction vanishes but rationality is unknown \cite[Example 1.5]{FJSVV}, see also Remark~\ref{rem:FJSVVexmp1.5}. Necessity of both conditions in part~\eqref{item:2nonnested-criterion} is in Remark~\ref{rem:2nonnested-criterion}.

As shown in \cite[Theorem 1.2]{FJSVV}, an underlying reason for the failure of the IJT obstruction to characterize rationality in this setting comes from the nontriviality of \(\Br\mathbb R\).
We study a certain quadratic twist to show that under certain assumptions, the IJT obstruction does characterize rationality. Specifically, the bitangents of a plane quartic \(\Delta=(f=0)\) are intimately related to the lines on the associated degree two del Pezzo surfaces \((t^2=f)\) and \((t^2=-f)\).
Work of Comessatti \cite{ComessattiFondamentiPL} shows that all the real bitangents split in one of these del Pezzo surfaces, and none of them split in the other, and this splitting is determined by the sign of the defining equation \(f\). For \(f=Q_1Q_3-Q_2^2\) the del Pezzo surface \((t^2=-f)\) is the branch locus of \eqref{eqn:DoubleCoverEquation}, and the image of its real points in \(\mathbb P^2\) is the locus where \(f \leq 0\). We use this to show that if \(Q_1Q_3-Q_2^2<0\) outside \(\Delta\) (which is precisely when the real bitangents split), then the IJT obstruction characterizes rationality (Proposition~\ref{prop:negative-outside-IJT}).
We apply this result to construct irrational examples with no topological obstruction to rationality.

In contrast, when \(Q_1Q_3-Q_2^2<0\) inside \(\Delta\), the IJT obstruction does not characterize rationality, as shown by the two non-nested ovals example of \cite[Theorem 1.3(1)]{FJSVV}. On our way to proving Theorem~\ref{thm:main}, several of the examples that we construct show that this failure persists for other real isotopy classes of \(\Delta\). Namely, we construct a \(\Delta(\mathbb R)=\emptyset\) example and a three ovals example where the IJT obstruction vanishes but the real locus of \(Y\) exhibits an obstruction to (stable) rationality (Proposition~\ref{prop:IJTvanishirrational}).
These three isotopy classes of \(\Delta\) are the only ones for which this is possible (Corollaries~\ref{cor:components-of-Y-and-Delta} and \ref{cor:two-nested-IJT}).

We also construct a rational example in each isotopy class of \(\Delta\). Since a real point on \(\Deltatilde\) gives a rationality construction for \(Y\) \cite[Proposition 4.1(5)]{FJSVV}, we primarily focus on the case when \(\Deltatilde(\mathbb R)=\emptyset\). In particular, in the two nested ovals case we give an example where \(\pi_1\) is surjective on real points and hence has a section by a result of Witt \cite{Witt-quadratic-forms}, but this section does not come from any known construction (Example~\ref{exmp:rational2nestedovals}\eqref{item:2nested-Deltatildeempty}).
We are not able to construct a similar example when \(\Delta\) is one oval, and we pose the following two questions:
\begin{question}[See Remark~\ref{rem:1oval-rational-question}]
Does there exist a rational threefold \(Y\) defined as in \eqref{eqn:DoubleCoverEquation} such that the real isotopy class of \(\Delta\) is one oval, \(Q_1Q_3-Q_2^2<0\) inside the oval, and \(\Deltatilde(\mathbb R)=\emptyset\)?
\end{question}

\begin{question}[See Corollary~\ref{cor:two-nested-IJT}]
If \(Y\) is as in \eqref{eqn:DoubleCoverEquation} and \(\Delta(\mathbb R)\) is two nested ovals, then rationality of \(Y\) \(\implies\) the IJT obstruction vanishes \(\implies\) \(Y(\mathbb R)\) is connected. Do any of the reverse implications hold?
\end{question}

\subsection{Outline}
In Section~\ref{sec:Preliminaries}, we review background and context for conic bundles over \(\mathbb P^2\) and rationality, key features of the double cover construction of \cite{FJSVV}, and the intermediate Jacobian torsor obstruction. We also make some observations relating the real topology of the double cover to that of the quartic curve. We prove Theorem~\ref{thm:DoubleCoverRationalityCriteria} in Section~\ref{sec:dP2}: we first study the associated degree \(2\) del Pezzo surface, and we then show the sufficiency of the intermediate Jacobian torsor obstruction in the case when \(Q_1Q_3-Q_2^2<0\) outside the ovals. In Section~\ref{sec:ExamplesConstruction} we apply our earlier results to give explicit examples of irrational and rational examples and prove Theorem~\ref{thm:main}.

\subsection*{Acknowledgements}
This research was conducted during the 2022 Research Experience for Undergraduates program at the University of Michigan Department of Mathematics, mentored by the first author. We are grateful to David Speyer and the University of Michigan math department for organizing the REU program and making this project possible. We thank Bianca Viray and Isabel Vogt for helpful discussions and comments, especially regarding Section~\ref{sec:dP2}; Sarah Frei, Brendan Hassett, and Soumya Sankar for helpful comments and conversations; and J\'anos Koll\'ar for the question that motivated this project.
The first author received support from NSF grant DMS-1840234, and the second author was supported by Karen Smith's NSF grant DMS-2101075.

\section{Preliminaries}\label{sec:Preliminaries}

We will first recall relevant definitions and background on rationality of conic bundle threefolds in Section~\ref{sec:ConicBundlesBg}. In Section~\ref{sec:DoubleCover}, we describe the double cover models \(Y\to\mathbb P^1\times\mathbb P^2\) of quartic conic bundles introduced in \cite{FJSVV}, which will be the models that we use throughout this paper. Next, in Section~\ref{sec:IJT-background} we review the intermediate Jacobian torsor obstruction and results from \cite{FJSVV} on the intermediate Jacobian torsors for the double cover models. In Section~\ref{sec:Yconnectedcomponents} we recall some facts about the real topology of even degree plane curves, and we make some observations relating \(Y(\mathbb R)\) and \(\Delta(\mathbb R)\).

\subsection{Rationality of standard conic bundles over \(\mathbb P^2\)}\label{sec:ConicBundlesBg}

We first review some preliminary notions about conic bundle threefolds and rationality. For more details on conic bundle threefolds, see \cite[Section 3]{prokhorov-rationality-conic-bundles}.

Let \(k\) be a field of characteristic \(\neq 2\).
A conic bundle over \(\mathbb P^2\) is a proper flat \(k\)-morphism \(\pi\colon X\to\mathbb P^2\) whose generic fiber is a smooth conic over \(\kk(\mathbb P^2)\). The \textsf{discriminant cover} \(\piDelta\colon \Deltatilde\to\Delta\) parametrizes the components of the singular fibers of \(\pi\). A conic bundle is \textsf{geometrically standard} if \(X\) is smooth and \(\rho(X_{\kbar}/\mathbb P^2_{\kbar})=1\).
The models we work with will have the property that \(\Delta\) is smooth and \(\pi\colon X\to\mathbb P^2\) is geometrically standard, and \(\piDelta\) is an \'etale double cover. In particular \(\pi\) has reduced fibers. We will introduce these models in Section~\ref{sec:DoubleCover}.

Let \(W\) be a smooth projective variety of dimension \(n\) over \(k\). Recall that \(W\) is said to be \textsf{rational over \(k\)} (or \textsf{\(k\)-rational}) if there is a birational map \(W\dashrightarrow\mathbb P^n\) defined over \(k\), \textsf{stably rational over \(k\)} if \(W\times\mathbb P^m\) is \(k\)-rational for some \(m\), and \textsf{unirational over \(k\)} if there is a dominant rational map \(\mathbb P^n\dashrightarrow W\) defined over \(k\).
If \(k\subset k'\) is a field extension, then \(k\)-(stable/uni-)rationality implies \(k'\)-(stable/uni-)rationality, but the converse need not hold, as demonstrated by a pointless real conic. We say that \(W\) is \textsf{geometrically rational} if the base change \(W_{\kbar}\) to the algebraic closure of \(k\) is \(\kbar\)-rational.

For the majority of this article, we work over \(\mathbb R\). As mentioned in the introduction, when we say that a variety is rational without specifying the ground field, we mean \(\mathbb R\)-rationality, \emph{not} \(\mathbb C\)-rationality.

In order to show that a variety is not rational, one must show that it has an obstruction to rationality. One obstruction is given by the Lang--Nishimura lemma, which implies that if \(W\) is \(k\)-rational (or even \(k\)-unirational), then it must contain a \(k\)-point. Over the real numbers, the locus \(W(\mathbb R)\) of real points also provides an obstruction to rationality: the number of real connected components is a birational invariant of smooth projective real varieties \cite[Theorem 3.4.12]{BCR-realAG}. So if \(W(\mathbb R)\) is disconnected, then \(W\) has an obstruction to stable rationality over \(\mathbb R\) (see also \cite{CTParimala} for an interpretation using unramified cohomology).

Over the complex numbers, rationality of conic bundles over \(\mathbb P^2\) is well understood. Namely, let \(X\to\mathbb P^2\) be a geometrically standard conic bundle with smooth discriminant curve \(\Delta\). Then \(X\) is \(\mathbb C\)-rational if and only if \(\deg\Delta\leq 4\), or \(\deg\Delta=5\) and \(\Deltatilde\to\Delta\) is defined by an even theta characteristic. Rationality is proven in the \(\deg\Delta\leq 4\) case using rationality results of Iskovskikh on conic bundle surfaces with low degree discriminant, by applying his surface classification to the generic fiber of a pencil of rational curves in \(\mathbb P^2\) \cite[Theorem 1]{Iskovskikh-rationality-cbs}. In the degree \(4\) case, one needs to blow down a divisor coming from a singular fiber of \(\pi\) to reduce to the degree \(3\) conic bundle surface case.
The higher degree results are due to the combined work of Tyurin, Masiewicki, Panin (\(\deg\Delta=5\)), and Beauville (\(\deg\Delta\geq 6\)); in the \(\mathbb C\)-irrational cases, \(X\) has an intermediate Jacobian obstruction to \(\mathbb C\)-rationality.
In addition, if \(\deg\Delta\leq 8\), then \(X\) is \(\mathbb C\)-unirational. We refer the reader to \cite[Theorem 9.1 and Corollary 14.3.4]{prokhorov-rationality-conic-bundles} for an overview of these results.

Over the real numbers, we recall \cite[Proposition 6.1]{FJSVV}, which in particular contains Theorem~\ref{thm:main}\eqref{item:Deltatildepoint_rational}.
If \(\deg\Delta\leq 3\), then \(X\) is rational if and only if \(X(\mathbb R)\neq\emptyset\) (e.g. this holds if \(\Delta(\mathbb R)\neq\emptyset\)): the Lang--Nishimura lemma shows necessity of an \(\mathbb R\)-point, and if \(X\) admits an \(\mathbb R\)-point then a modification of the proof over \(\mathbb C\) shows that \(X\) is rational. In degree \(4\), the proof of geometric rationality does not always descend, even if \(X(\mathbb R)\neq\emptyset\), because the singular fibers of \(\pi\) need not be split over \(\mathbb R\). When \(\Deltatilde(\mathbb R)\neq\emptyset\), however, the argument over \(\mathbb C\) goes through if the pencil is chosen through the image of a point of \(\Deltatilde(\mathbb R)\). Similarly, if \(X\) has an \(\mathbb R\)-point away from \(X_\Delta := X\times_{\mathbb P^2} \Delta\), then \(X\) is unirational by a modification of the argument over \(\mathbb C\). (These results hold more generally over any field of characteristic \(\neq 2\), see \cite[Section 6.1]{FJSVV}.)

\subsection{Conic bundle threefolds realized as double covers of \(\mathbb P^1\times\mathbb P^2\)}\label{sec:DoubleCover}
We recall the following models of conic bundles, which were studied by Frei--Ji--Sankar--Viray--Vogt \cite{FJSVV}.
These are the models that we will study throughout the paper, and we will work over the real numbers. First, we recall a result of Bruin that allows us express \'etale double covers of smooth plane quartics in a particular form.
\begin{thm}[{\cite[Section 3]{bruin}}]\label{thm:bruin}
Let \(\piDelta\colon\Deltatilde\to\Delta\) be an \'etale double cover of a smooth plane quartic. Then there exist quadratic forms \(Q_1,Q_2,Q_3\in \mathbb R[u,v,w]\) such that \(\Deltatilde\to\Delta\) is of the form
\begin{equation}\begin{gathered}\label{eqn:DeltatildeDelta}\Delta=(Q_1Q_3-Q_2^2=0), \\ \Deltatilde=(Q_1-r^2=Q_2-rs=Q_3-s^2=0).\end{gathered}\end{equation}
\end{thm}
Define the double cover \(\piDoubleCover\colon Y\to\mathbb P^1_{[t_0:t_1]}\times\mathbb P^2_{[u:v:w]}\): \begin{equation}\label{eqn:DoubleCover}z^2=t_0^2 Q_1+2t_0t_1 Q_2+t_1^2Q_3.\end{equation}
The second projection \(\pi_2\colon Y\to\mathbb P^2\) is a conic bundle whose discriminant double cover is defined by \eqref{eqn:DeltatildeDelta}.
The isomorphism class of \(Y\) only depends on the double cover \(\Deltatilde\to\Delta\), not on the choice of the quadrics \(Q_i\) (see \cite[Section 4]{FJSVV}), and so we will denote the double cover given above by \(Y_{\Deltatilde/\Delta}\), or by \(Y\) when the context is clear. We review the following properties of \(Y\).

\begin{prop}[{\cite[Theorem 2.6, Propositions 4.1 and 4.3]{FJSVV}}]\label{prop:Yproperties}
If \(Y\) is defined as in Equation~\eqref{eqn:DoubleCover}, then:
\begin{enumerate}
    \item\label{item:Yconicbundle} \(Y\) is a smooth Fano threefold, and the second projection \(\pi_2\colon Y\to\mathbb P^2\) is a geometrically standard conic bundle with discriminant cover \(\piDelta\colon\Deltatilde\to\Delta\). In particular, \(Y\) is \(\mathbb C\)-rational, and if a smooth fiber of \(\pi_2\) contains a real point then \(Y\) is unirational.
    \item\label{item:Yquadricsurfacebundle} The first projection \(\pi_1\colon Y\to\mathbb P^1\) is a quadric surface bundle. In particular, if \(\pi_1\) has a (real) section, then \(Y\) is rational.
    \item\label{item:Deltatilde-point-section} If \(\Deltatilde(\mathbb R)\neq\emptyset\), then \(\pi_1\) has a (real) section.
    \item\label{item:Gamma-curve} The Stein factorization of the relative variety of lines is \(\mathcal F_1(Y/\mathbb P^1)\to\Gamma\to\mathbb P^1\), where \(M_i\) is the symmetrix \(3\times 3\) matrix corresponding to \(Q_i\) and \(\Gamma\) is the genus \(2\) curve defined by \[y^2=-\det(t^2 M_1+2tM_2+M_3).\]
    \item\label{item:alluptoconstantBrclass} \cite[Section 6.9.6]{Poonen-Q-points} If \(X\to\mathbb P^2\) is a geometrically standard conic bundle with discriminant cover \(\Deltatilde\to\Delta\), then \([(Y_{\Deltatilde/\Delta})_\eta]-[X_\eta]\in\mathrm{Im}(\Br\mathbb R\to\Br\kk(\mathbb P^2))\).
\end{enumerate}
\end{prop}
In particular, Theorem~\ref{thm:bruin} and Proposition~\ref{prop:Yproperties}\eqref{item:alluptoconstantBrclass} imply that, up to a constant Brauer class, any geometrically standard conic bundle over \(\mathbb P^2\) with smooth quartic discriminant curve is birational over \(\mathbb P^2\) to one of the form~\eqref{eqn:DoubleCover}.

\begin{rem}
Proposition~\ref{prop:Yproperties}\eqref{item:Deltatilde-point-section} shows that a point of \(\Deltatilde\) gives rise to a section of \(\pi_1\). However, not every section of \(\pi_1\) arises in this way: the rational examples constructed in the proof of Theorem~\ref{thm:main}\eqref{item:0_oval}--\eqref{item:4_ovals} all have \(\Deltatilde(\mathbb R)=\emptyset\) and admit sections of \(\pi_1\) over \(\mathbb R\). In Section~\ref{sec:dP2}, we will see another source of sections of \(\pi_1\).
\end{rem}

To each \'etale double cover \(\Deltatilde\to\Delta\), we also associate a twisted double cover of \(\mathbb P^1\times\mathbb P^2\).
\begin{defn}\label{defn:TwistedDoubleCover}
Let \(Q_1,Q_2,Q_3\) and \(\Deltatilde\to\Delta\) be as in Theorem~\ref{thm:bruin}. The \textsf{twisted double cover} \(Y_{\Deltabar/\Delta}\to\mathbb P^1\times\mathbb P^2\) associated to \(\Deltatilde\to\Delta\) is defined by the equation
\[z^2=-t_0^2 Q_1+2t_0t_1 Q_2-t_1^2Q_3.\]
\end{defn}
Since this double cover is of the form in Equation~\eqref{eqn:DoubleCover} obtained by replacing \(Q_1\) and \(Q_3\) with \(-Q_1\) and \(-Q_3\), the threefold \(Y_{\Deltabar/\Delta}\) satisfies the properties of Proposition~\ref{prop:Yproperties} (with the appropriate substitutions). In particular, \(Y_{\Deltabar/\Delta}\to\mathbb P^2\) is a conic bundle with discriminant cover \(\piDeltabar\colon \Deltabar\to\Delta\), where \[\Deltabar=(Q_1+r^2=Q_2-rs=Q_3+s^2=0)\] is the quadratic twist of \(\Deltatilde\).

The branch locus of \(\Gamma\) in Proposition~\ref{prop:Yproperties}\eqref{item:Gamma-curve} is defined by the equation \(-\det(t^2M_1+2tM_2+M_3)\), which gives the singular fibers of the quadric surface fibration \(\pi_1\). The signature of the \(4\times 4\) matrix
\[\begin{pmatrix}[c|c] t^2M_1+2tM_2+M_3 & 0 \\ \hline 0 & -1 \end{pmatrix}\]
corresponding to the quadric surface \(Y_{[t:1]}\) is constant on each interval of \(\mathbb P^1(\mathbb R)\) away from the real branch points, and at each real branch point of \(\Gamma\) the number of positive eigenvalues changes by \(\pm 1\).

Note that \(Y_{[t:1]}(\mathbb R)\neq\emptyset\) if and only if the matrix corresponding to \(Y_{[t:1]}\) is indefinite. By Witt's Decomposition Theorem \cite[Section 8]{EKM}, \(Y_{[t:1]}\) contains lines defined over \(\mathbb R\) if and only if \(Y_{[t:1]}\) has signature \((2,2)\).
We will also repeatedly use the following result of Witt to construct sections:

\begin{thm}[{\cite[Satz 22]{Witt-quadratic-forms}}]\label{thm:witt}
Let \(C\) be a smooth projective real curve and \(\pi\colon X\to C\) a quadric bundle of relative dimension \(m\geq 1\). If the induced map \(\pi(\mathbb R)\colon X(\mathbb R)\to C(\mathbb R)\) on real points is surjective, then \(\pi\) has a section defined over \(\mathbb R\).
\end{thm}

\subsection{The intermediate Jacobian torsor obstruction to rationality}\label{sec:IJT-background}

The classical intermediate Jacobian obstruction to \(\mathbb C\)-rationality, introduced by Clemens--Grifiths in their proof of the irrationality of the cubic threefold \cite{Clemens-Griffiths-ij}, states that the intermediate Jacobian of a \(\mathbb C\)-rational threefold must be isomorphic to a product of Jacobians of curves. Over non-closed fields, Hassett--Tschinkel \cite[Section 11.5]{ht-intersection-quadrics} \cite[Sections 3 and 4]{HT-cycle} and Benoist--Wittenberg \cite[Theorem 3.11]{bw-ij} have recently introduced a refinement of this obstruction involving the torsors over the intermediate Jacobian. Assuming for simplicity that the intermediate Jacobian is isomorphic to \(\bPic^0_{\Gamma}\) for a curve \(\Gamma\) of genus \(\geq 2\), their refinement states that moreover for each Galois-invariant geometric curve class, the corresponding torsor over the intermediate Jacobian must be isomorphic to some \(\bPic^i_{\Gamma}\). (See \cite[Theorem 3.11]{bw-ij} for the general statement.) Following \cite{FJSVV}, we refer to this as the \textsf{intermediate Jacobian torsor (IJT) obstruction}.

Since we work over \(\mathbb R\), when we mention the IJT obstruction we always mean the obstruction over \(\mathbb R\).

The IJT obstruction has been used to great effect to study rationality of geometrically rational threefolds over non-closed fields: Hassett--Tschinkel (over \(\mathbb R\)) \cite{ht-intersection-quadrics} and Benoist--Wittenberg (over arbitrary fields) \cite{bw-ij} showed this obstruction characterizes rationality for smooth complete intersections of two quadrics in \(\mathbb P^5\), and Kuznetsov--Prokhorov used it for several cases of their classification of rationality for prime Fano threefolds \cite{KP-Fano}.
However, Benoist--Wittenberg also showed that the IJT obstruction is not sufficient to characterize rationality by constructing an example of a (non geometrically standard) real conic bundle threefold \(X\to S\) whose intermediate Jacobian is trivial but such that \(S(\mathbb R)\) is disconnected; hence, the IJT obstruction vanishes but \(X\) has a Brauer obstruction to (stable) rationality over \(\mathbb R\) \cite[Theorem 5.7]{bw-cg}.

In \cite{FJSVV}, Frei--Ji--Sankar--Viray--Vogt studied the intermediate Jacobian torsors for geometrically standard conic bundle threefolds, giving an explicit desciption using certain torsors over the Prym variety of the discriminant cover \(\Deltatilde\to\Delta\) \cite[Theorem 1.1]{FJSVV}. For the double covers described in Section~\ref{sec:DoubleCover}, \cite[Theorem 1.2]{FJSVV} and \cite[Section 5]{bruin} show that the intermediate Jacobian of \(Y\) is \(P:= \Prym_{\Deltatilde/\Delta}\cong \bPic^0_{\Gamma}\), where \(\Gamma\) is the genus two curve defined in Proposition~\ref{prop:Yproperties}\eqref{item:Gamma-curve}. Moreover, in this setting \cite[Theorem 4.4]{FJSVV} gives an extended description of the torsors. In particular, there are four intermediate Jacobian torsors
\[P\sqcup \Ptilde = \piDelta_*^{-1}[\mathcal O_\Delta] \subset\bPic^0_{\Deltatilde} \quad \text{ and }\quad \Pm{1} \sqcup \Ptildem{1} = \piDelta_*^{-1} [\mathcal O_\Delta(1)] \subset\bPic^4_{\Deltatilde}\] satisfying \(\Ptilde+\Pm{1}=\Ptildem{1}\) as \(P\)-torsors, and \(\Pm{1}\cong\bPic^1_\Gamma\).
(Here \(\piDelta_*\colon\bPic_{\Deltatilde}\to\bPic_\Delta\) is the norm map.)

\(\bPic^1_\Gamma(\bb R)\neq\emptyset\) by a theorem of Lichtenbaum (see \cite[paragraph after Corollary 4]{Poonen-Stoll}). 
So the IJT obstruction vanishes if and only if \(\Ptildem{1}(\bb R)\neq\emptyset\).

\cite{FJSVV} showed that a point on \(\Ptildem{1}\) gives a Galois-invariant geometric section of the quadric surface bundle \(\pi_1\colon Y\to\mathbb P^1\), coming from a geometric section of a conic bundle over the corresponding line in \(\mathbb P^2\) \cite[Proposition 4.5]{FJSVV}. However, in general this geometric section need not descend; thus, they show that the failure of the IJT obstruction for the double covers of Section~\ref{sec:DoubleCover} comes from the nontriviality of the 2-torsion in the Brauer group of the ground field. (Indeed, when \(\Br k[2]=0\), the IJT obstruction does characterize rationality \cite[Theorem 1.4]{FJSVV}.) \cite[Theorem 1.3(1)]{FJSVV} exploits the nontriviality of \(\Br\mathbb R\) to give an example of an irrational conic bundle with \(\Ptildem{1}(\mathbb R)\neq\emptyset\). (In fact \cite[Theorem 1.3(1)]{FJSVV} is defined over \(\mathbb Q\) and has \(\Ptildem{1}(\mathbb Q)\neq\emptyset\). Their example is irrational over \(\mathbb R\), and its reduction modulo \(p\) is \(\mathbb F_p\)-rational for every \(p\neq 2\) such that \(\Delta_p\) and \(\Deltatilde_p\) are smooth.) For this construction, they gave the following geometric interpretation \cite[Section 2]{FJSVV} of the real points of \(\Ptildem{1}\) as \(\Gal(\mathbb C/\mathbb R)\)-invariant sets of four points \(p_1,p_2,p_3,p_4\in\Deltatilde(\mathbb C)\) such that:
\begin{enumerate}
\item
\(p_1,p_2,p_3,p_4\) does not span a \(2\)-plane in \(\mathbb P^4\), and
\item
\(\piDelta_*(p_1+p_2+p_3+p_4)=\Delta\cap\ell\) for a real line \(\ell\subset\mathbb P^2\). (If \(\Deltatilde(\mathbb R)=\emptyset\) then \(\ell\) does not meet \(\Delta\) transversely in any real points \cite[Lemma 5.1]{FJSVV}.)
\end{enumerate}

Note that by lower semicontinuity of rank, the property that \(\Ptildem{1}\) has an \(\mathbb R\)-point is an open condition.

\subsection{Real connected components of \(Y\) and \(\Delta\)}\label{sec:Yconnectedcomponents}
In this section, we make some observations about the real connected components of \(Y\) and the real isotopy class of the discriminant curve \(\Delta\).

For a morphism \(\phi\colon V\to W\) of quasi-projective algebraic varieties over \(\mathbb R\), we let \(\phi(\mathbb R)\colon V(\mathbb R)\to W(\mathbb R)\) denote the induced map of topological spaces on the sets of real points (with the Euclidean topology).

\begin{notation}\label{notation:ovals}
If \(f\in\mathbb R[u,v,w]\) is a homogeneous polynomial defining a smooth curve of even degree, then the sign of \(f(P)\) for \(P\in\mathbb P^2(\mathbb R)\) is well defined. We denote by \((f>0)_{\mathbb R}\) the set of real points for which \(f(P)>0\) (similarly for \(\geq,=,\leq\), and \(<\)). Every connected component of \((f=0)_{\mathbb R}\) is an oval, and the complement of \((f=0)_{\mathbb R}\) in \(\mathbb P^2(\mathbb R)\) is a disjoint union of a non-orientable set \(U_f\) and a finite number of discs \cite[Section 2.7]{Mangolte-realAG}. The non-orientable set \(U_f\) is the \textsf{outside} of the curve defined by \(f\).
\end{notation}

In the case where \(f\) defines a smooth quartic curve \(\Delta\), Zeuthen \cite{Zeuthen1874} proved the following classification result for the real isotopy class of \(\Delta\). (Recall that \(\Delta\) has 28 complex bitangents.)

\begin{tabular}{ c | c | c | c | c | c | c}
 \(\Delta(\mathbb R)\) & \(\emptyset\) & One oval & Two nested ovals & Two non-nested ovals & Three ovals & Four ovals\\ 
 \hline
 Real bitangents & \(4\) & \(4\) & \(4\) & \(8\) & \(16\) & \(28\)
\end{tabular}

Sections~\ref{sec:IrrationalExamples} and \ref{sec:RationalExamples} will contain figures illustrating all non-empty real isotopy classes. We will sometimes denote the real locus of the plane curve \(\Delta\) by \((\Delta =0)_{\mathbb R}:= \Delta(\mathbb R)\), and we will denote the outside of \(\Delta\) by \(U_\Delta\).

Next, we use the maps \(\pi_1(\mathbb R)\) and \(\pi_2(\mathbb R)\) to relate the real connected components of \(Y\) and the real connected components of the quartic \(\Delta\).

\begin{lem}\label{lem:ImComponents}
Let \(Y:= Y_{\Deltatilde/\Delta}\) be as defined in Section~\ref{sec:DoubleCover}.
The number of connected components of \(Y(\mathbb R)\) is equal to the number of connected components of its image under \(\pi_i\colon Y\to\mathbb P^i\) for \(i=1,2\).
\end{lem}

\begin{proof}
Since \(\pi_i\colon Y\to\mathbb P^i\) is the finite morphism \(\piDoubleCover\) composed with the projection \(\mathbb P^1\times\mathbb P^2\to\mathbb P^i\), it follows from \cite[Theorem 4.2]{DelfsKnebusch} and compactness of \(\mathbb P^n(\mathbb R)\) that \(\pi_i(\mathbb R)\) is a continuous closed map.
The claim then holds since the fibers of \(\pi_i\) are positive-dimensional quadrics and in particular have connected real loci.
\end{proof}

\begin{lem}\label{lem:Yleq3components}
If \(Y:= Y_{\Deltatilde/\Delta}\) is as defined in Section~\ref{sec:DoubleCover}, then
\(Y(\mathbb R)\) has at most three connected components.

\end{lem}

\begin{proof}
By Lemma~\ref{lem:ImComponents}, it suffices to show that the image of \(\pi_1(\mathbb R)\) has at most three components. The signature of the fibers of \(\pi_1\) can only change at the real branch points of the genus two curve \(\Gamma\) defined in Proposition~\ref{prop:Yproperties}\eqref{item:Gamma-curve}, so the number of connected components of \(\pi_1(\mathbb R)\) is at most half the number of real branch points and so is at most \(\tfrac{1}{2}\cdot 6=3\).
\end{proof}

\begin{lem}\label{lem:Impi2}
In the setting of Section~\ref{sec:DoubleCover}, the image of \(\pi_2(\mathbb R)\) is \((Q_1\geq 0)_{\mathbb R}\cup(Q_1Q_3-Q_2^2\leq 0)_{\mathbb R} \subseteq\mathbb P^2(\mathbb R)\).
\end{lem}

\begin{proof}
The fiber of \(\pi_2\) above \(P\in\mathbb P^2(\mathbb R)\) is the conic corresponding to the symmetric matrix
\[\begin{pmatrix} Q_1(P) & Q_2(P) & 0 \\ Q_2(P) & Q_3(P) & 0 \\ 0 & 0 & -1 \end{pmatrix}\]
so the fiber contains an \(\mathbb R\)-point if and only if the top \(2\times 2\) submatrix is not negative definite. By Sylvester's criterion, this submatrix is negative definite if and only if \(Q_1(P)<0\) and \((Q_1Q_3-Q_2^2)(P)>0\).
\end{proof}

From Proposition~\ref{prop:Yproperties}\eqref{item:Yconicbundle} and Lemma~\ref{lem:Impi2}, it immediately follows that:
\begin{cor}
If \(Y_{\Deltatilde/\Delta}(\mathbb R)=\emptyset\), then \(\Delta(\mathbb R)=\emptyset\). If \(\Delta(\mathbb R)\neq\emptyset\), then \(Y_{\Deltatilde/\Delta}\) is unirational.
\end{cor}

\begin{cor}\label{cor:components-of-Y-and-Delta}
If \(Y_{\Deltatilde/\Delta}(\mathbb R)\) is disconnected, then \(\Delta(\mathbb R)\) must be two or three ovals. More precisely:
\begin{enumerate}
\item If \(Y_{\Deltatilde/\Delta}(\mathbb R)\) has three connected components, then \(\Delta(\mathbb R)\) is three ovals; and
\item If \(Y_{\Deltatilde/\Delta}(\mathbb R)\) has two connected components, then \(\Delta(\mathbb R)\) is two non-nested ovals or two nested ovals.\end{enumerate}
\end{cor}

\begin{proof}
If \(Q_1\) is positive definite, then the image of \(\pi_2(\mathbb R)\) is \(\mathbb P^2(\mathbb R)\) by Lemma~\ref{lem:Impi2}, so we may assume that \(Q_1\) is negative definite or indefinite.
First suppose \(Q_1\) is negative definite. Then the image of \(\pi_2(\mathbb R)\) is \((Q_1Q_3-Q_2^2\leq 0)_{\mathbb R}\), which can only be disconnected if \(\Delta(\mathbb R)\) is two or more ovals. If \((Q_1Q_3-Q_2^2\leq 0)_{\mathbb R}\) is disconnected, then it has the same number of connected components as \(\Delta(\mathbb R)\) and, by Lemma~\ref{lem:ImComponents}, it also has the same number of connected components as \(Y(\mathbb R)\). So by Lemma~\ref{lem:Yleq3components}, \(\Delta(\mathbb R)\) is either two or three ovals.

It remains to consider the case when \(Q_1\) is indefinite, so its real locus is one oval. Since \((Q_1=0)_{\mathbb R}\subset(Q_1Q_3-Q_2^2\leq 0)_{\mathbb R}\), we have that \((Q_1\geq 0)_{\mathbb R}\cup (Q_1Q_3-Q_2^2\leq 0)_{\mathbb R}\) is either equal to \((Q_1Q_3-Q_2^2\leq 0)_{\mathbb R}\) or all of \(\mathbb P^2(\mathbb R)\). Thus, again using Lemma~\ref{lem:Yleq3components} to rule out the four ovals case when \(Y(\mathbb R)\) is disconnected, we conclude that \(Y(\mathbb R)\) is disconnected if and only if \(\Delta(\mathbb R)\) is either two or three ovals, and that in the disconnected case \(Y(\mathbb R)\) and \(\Delta(\mathbb R)\) have the same number of connected components.
\end{proof}

\begin{rem}
All cases in Corollary~\ref{cor:components-of-Y-and-Delta} occur; see Section~\ref{sec:IrrationalExamples} and \cite[Theorem 1.3(1)]{FJSVV}.
\end{rem}

\begin{rem}\label{rem:2,2sig-nested}
If \(Y_{\Deltatilde/\Delta}(\mathbb R)\) is disconnected and \(\pi_1\) has a fiber with signature \((2,2)\), then the real isotopy class of \(\Delta\) is two nested ovals. Indeed, Corollary~\ref{cor:components-of-Y-and-Delta} and the fact that \(\Gamma\) has at most six real Weierstrass points imply that the image of \(Y_{\Deltatilde/\Delta}(\mathbb R)\) in \(\mathbb P^2(\mathbb R)\) is \((Q_1Q_3-Q_2^2\leq 0)_{\mathbb R}\) and that \(\Delta(\mathbb R)\) consists of two ovals. After a coordinate change on \(\mathbb P^1\) \cite[Theorem 2.6]{FJSVV} we may assume \(Q_1\) has signature \((2,1)\). Recalling that \(U_{Q_1}\) denotes the outside of the plane conic \(Q_1\), the signature assumption on \(Q_1\) implies \(U_{Q_1}=(Q_1>0)_{\mathbb R}\). Since \(U_{Q_1}\) is not orientable it cannot be contained in a disc, so \(U_{Q_1} \subset (Q_1Q_3-Q_2^2\leq 0)_{\mathbb R}\) implies that \(U_\Delta\) is one of the two connected components of \((Q_1Q_3-Q_2^2< 0)_{\mathbb R}\), which implies the two ovals of \(\Delta\) must be nested.
\end{rem}

We now relate the real points of \(\Deltatilde\) to those of the corresponding curve on the the twisted double cover.

\begin{lem}\label{lem:pi1surjDeltatildeEmpty}
Let \(\Deltatilde\to\Delta\) and \(\Deltabar\to\Delta\) be as defined in Section~\ref{sec:DoubleCover}. Then
\[ \Delta(\mathbb R)=\mathrm{Im}\left( \piDelta(\mathbb R)\colon\Deltatilde(\mathbb R)\to\Delta(\mathbb R) \right) \sqcup \mathrm{Im} \left(\piDeltabar(\mathbb R)\colon\Deltabar(\mathbb R)\to\Delta(\mathbb R)\right).
\]
In particular, the map
\(\piDelta(\mathbb R)\colon\Deltatilde(\mathbb R)\to\Delta(\mathbb R)\) on real points is surjective if and only if \(\Deltabar(\mathbb R)=\emptyset\).
\end{lem}

\begin{proof}
First, we note that the real points of \(\Deltatilde\) lie over the locus \((\Delta=0)_{\mathbb R}\cap(Q_1\geq 0)_{\mathbb R}\), and the real points of \(\Deltabar\) lie over \((\Delta=0)_{\mathbb R}\cap(-Q_1\geq 0)_{\mathbb R}\). From the equations \eqref{eqn:DeltatildeDelta} we see that
\begin{equation}\label{eqn:image-of-twist}
(\Delta=0)_{\mathbb R}\cap(Q_1> 0)_{\mathbb R}\subset\mathrm{Im}\piDelta(\mathbb R) \quad \text{ and } \quad (\Delta=0)_{\mathbb R}\cap(-Q_1> 0)_{\mathbb R}\subset\mathrm{Im}\piDeltabar(\mathbb R).\end{equation}
The intersection \((Q_1=0)_{\mathbb R} \cap(\Delta=0)_{\mathbb R}\) is at most a finite number of points, since otherwise \(\Delta\) is not smooth.
Furthermore, since \((Q_1=0)_{\mathbb R}\) is connected and is contained in \((Q_1Q_3-Q_2^2\leq 0)_{\mathbb R}\), it cannot cross \((\Delta=0)_{\mathbb R}\), so on each oval of \(\Delta(\mathbb R)\) either \(Q_1\geq 0\) or \(-Q_1\geq 0\). Thus each connected component from each set in \eqref{eqn:image-of-twist} above is an oval of \(\Delta(\mathbb R)\) minus a finite number of points, and each oval of \(\Delta(\mathbb R)\) contains points in one of the two sets in \eqref{eqn:image-of-twist}.
Both \(\Deltatilde\) and \(\Deltabar\) are smooth projective curves, so their real loci are homeomorphic to a (possibly empty) disjoint union of circles \cite[Section 3.3]{Mangolte-realAG}. Since \(\piDelta\) and \(\piDeltabar\) are finite morphisms, the induced maps \(\piDelta(\mathbb R)\) and \(\piDelta(\mathbb R)\) are closed by \cite[Theorem 4.2]{DelfsKnebusch}. Therefore, the closure (in \(\Delta(\mathbb R)\)) of each component in the lefthand set in \eqref{eqn:image-of-twist} is in the image of \(\piDelta(\mathbb R)\), and the closure of each component in the righthand set in \eqref{eqn:image-of-twist} is in the image of \(\piDeltabar(\mathbb R)\).
\end{proof}

It follows that if \(\Delta(\mathbb R)\neq\emptyset\) and \(\piDelta(\mathbb R)\) is not surjective, then \(\Deltabar\) has an \(\mathbb R\)-point and so \(Y_{\Deltabar/\Delta}\) is rational. In particular, when \(\Delta(\mathbb R)\neq\emptyset\) we have the following:
\begin{cor}\label{cor:YorTwistRational}
If \(\Delta(\mathbb R)\neq\emptyset\), then at least one of \(Y_{\Deltatilde/\Delta}\) or \(Y_{\Deltabar/\Delta}\) is rational.
\end{cor}
However, when \(\Delta(\mathbb R)=\emptyset\) it is possible for both \(Y_{\Deltatilde/\Delta}\) and \(Y_{\Deltabar/\Delta}\) to be irrational, see Example~\ref{exmp:0ovalIJT}.

\section{The degree \(2\) del Pezzo surface and the intermediate Jacobian torsor obstruction}\label{sec:dP2}

In this section, we prove Theorem~\ref{thm:DoubleCoverRationalityCriteria}, which characterizes rationality for the double covers in Section~\ref{sec:DoubleCover} for all but two isotopy classes of the quartic \(\Delta\).
The key technical input to Theorem~\ref{thm:DoubleCoverRationalityCriteria}\eqref{item:0ovals-criterion}--\eqref{item:2nonnested-criterion} is Proposition~\ref{prop:negative-outside-IJT}, which shows that under an assumption on the sign of the equation \(Q_1Q_3-Q_2^2\) defining \(\Delta\), the intermediate Jacobian torsor obstruction characterizes rationality.

Before proving Proposition~\ref{prop:negative-outside-IJT}, we first show Theorem~\ref{thm:DoubleCoverRationalityCriteria}\eqref{item:3ovals-rationality-criterion}--\eqref{item:4ovals-rational}. Namely, we show that \(Y_{\Deltatilde/\Delta}\) is rational if \(\Delta(\mathbb R)\) is four ovals, and that if \(\Delta(\mathbb R)\) is three ovals then topological criterion that \(Y_{\Deltatilde/\Delta}(\mathbb R)\) is connected is sufficient to guarantee rationality.

\begin{prop}[Theorem~\ref{thm:DoubleCoverRationalityCriteria}\eqref{item:3ovals-rationality-criterion}--\eqref{item:4ovals-rational}]\label{prop:4ovals_rational}
Let \(Y:= Y_{\Deltatilde/\Delta}\to\mathbb P^1\times\mathbb P^2\) be a double cover constructed as in Section~\ref{sec:DoubleCover}. Assume that either
\begin{enumerate}
    \item The real isotopy class of \(\Delta\) is four ovals, or
    \item The real isotopy class of \(\Delta\) is three ovals and \(Y(\mathbb R)\) is connected.
\end{enumerate}
Then the quadric surface bundle \(Y\to\mathbb P^1\) admits a section (over \(\mathbb R\)). In particular, \(Y\) is rational (over \(\mathbb R\)).
\end{prop}

To prove Proposition~\ref{prop:4ovals_rational}, we first study an associated degree two del Pezzo surface and use the geometry of this surface to construct certain geometric sections of the quadric surface bundle \(\pi_2 \colon Y \to \mathbb P^1\) in Section~\ref{sec:dP2-lines}.  In Section~\ref{sec:UDelta} we show that under the assumptions of Proposition~\ref{prop:4ovals_rational}, such a section exists over \(\mathbb R\).

Then, in Section~\ref{sec:IJT-sufficient}, to prove Proposition~\ref{prop:negative-outside-IJT} we show that whenever this degree two del Pezzo surface contains real lines, the IJT obstruction characterizes rationality of \(Y\). This condition that the surface contains real lines is determined by the sign of \(Q_1Q_3-Q_2^2\).

Throughout, we let \(Q_1,Q_2,Q_3\in\mathbb R[u,v,w]\) be quadratic forms defining an \'etale double cover of a smooth quartic as in Theorem~\ref{thm:bruin}, and \(Y:= Y_{\Deltatilde/\Delta}\) the associated double cover of \(\mathbb P^1\times\mathbb P^2\) defined in Section~\ref{sec:DoubleCover}.

\subsection{Lines on the associated degree \(2\) del Pezzo surface}\label{sec:dP2-lines}
The results in this section are based on ideas joint with S. Frei, S. Sankar, B. Viray, and I. Vogt. In particular, Proposition~\ref{prop:W-lines} and the idea to use lines on the associated degree \(2\) del Pezzo surface to construct sections of the quadric surface bundle were obtained during the preparation of \cite{FJSVV}.

Recall from Section~\ref{sec:DoubleCover} that the double cover \(Y := Y_{\Deltatilde/\Delta} \to\mathbb P^1_{[t_0:t_1]}\times\mathbb P^2_{[u:v:w]}\) is defined by the equation \(z^2=t_0^2 Q_1(u,v,w)+2t_0t_1 Q_2(u,v,w)+t_1^2 Q_3(u,v,w)\). The branch locus is the \((2,2)\)-divisor \begin{equation}\label{eqn:dP2}W:= W_{\Deltatilde/\Delta}:= \left(t_0^2 Q_1(u,v,w)+2t_0t_1 Q_2(u,v,w)+t_1^2 Q_3(u,v,w)=0\right).\end{equation} Let \(\pi'_i\colon W\hookrightarrow\mathbb P^1\times\mathbb P^2\to\mathbb P^i\) denote the compositions of the inclusion with the projections.

The second projection \(\pi'_2\colon W\to\mathbb P^2\) is a double cover branched along the quartic curve \(\Delta\), so \(W\) is a del Pezzo surface of degree two. Thus, \(W_{\mathbb C}\) is isomorphic to the blow up of \(\mathbb P^2_{\mathbb C}\) at seven points \(P_1,\ldots,P_7\in\mathbb P^2(\mathbb C)\) in general position. We use this description to index the \(56\) (complex) lines (i.e.  genus \(0\) curves with self-intersection \(-1\)) of \(W_{\mathbb C}\),  which map to the 28 (complex) bitangents of \(\Delta\):
the exceptional divisors \(E_1,\ldots,E_7\) of the blow up;
the strict transforms \(L_{ij}\) of the line passing through \(P_i\) and \(P_j\) for \(i< j\);
the strict transforms \(Q_{ij}\) of the conic passing through the five points complementary to \(P_i,P_j\) for \(i< j\); and
the strict transforms \(C_i\) of the cubic passing through \(P_1,\ldots,P_7\) with multiplicity two at \(P_i\). (See e.g. \cite[Section 8.7]{Dolgachev-CAG}.)

Away from the conic defined by \(Q_1\), the double cover \(\pi'_2\colon W\to\mathbb P^2\) is locally given by the double cover
\begin{equation}\label{eq:localW}
( t Q_1 + Q_2)^2 =-(Q_1Q_3-Q_2^2).
\end{equation}

The first projection \(\pi'_1\colon W\to\mathbb P^1\) is a conic bundle whose discriminant divisor is equal to the branch locus \(-\det(t_0^2M_1+2t_0t_1M_2+t_1^2M_3)=0\) of the genus two curve \(\Gamma\) defined in Section~\ref{sec:DoubleCover}. Each singular fiber of \(\pi'_1\) is a rank \(2\) conic, so the components of the singular fibers of \(\pi'_1\) make up twelve of the lines on \(W_{\mathbb C}\).

The lines on \(W_{\mathbb C}\) come in pairs \((\ell,\ell')\) with \(\ell \cdot \ell' = 2\). Each pair of lines maps under \(\pi'_2\) to the same bitangent of \(\Delta\). In particular, \(\ell\) is defined over \(\mathbb R\) if and only if \(\ell'\) is. Using the above description of the lines after identifying \(W_{\mathbb C}\) with a blow up of \(\mathbb P^2\), the line pairs are \((E_i, C_i)\) and \((L_{ij}, Q_{ij})\).

\begin{prop}\label{prop:W-lines}
The 56 geometric lines on \(W_{\mathbb C}\) have the following decomposition into three sets:
\begin{enumerate}
\item\label{item:0-intersection}
12 are the geometric components of the six singular fibers of the conic bundle \(\pi_1'\);
\item\label{item:2-intersection}
12 give degree two geometric multisections of \(\pi_1'\); and
\item\label{item:1-intersection}
The remaining 32 give geometric sections of \(\pi'_1\), and hence give geometric sections of the quadric surface bundle \(\pi_1 \colon Y \to \mathbb P^1\).
\end{enumerate}
Moreover, the sets \eqref{item:0-intersection} and \eqref{item:2-intersection} each contain the same even number of lines defined over \(\mathbb R\).
\end{prop}

\begin{proof}
Throughout this proof, we work over \(\mathbb C\). The result will be proven by computing intersection numbers.
By the action of the Weyl group \(W(E_7)\) on the lines of the del Pezzo surface, we may assume that one of the singular fibers of \(\pi'_1\) is \(E_1+L_{12}\) \cite[Section 5]{FriedmanMorgan}. Then one computes:
\begin{enumerate}
    \item\label{item:bad-fibers} The lines \(E_1, \{E_i\mid i\geq 3\}, L_{12}, \{L_{2j} \mid j\geq 3\}\) each have intersection \(0\) with \(E_1+L_{12}\). These are the twelve components of the singular fibers of \(\pi'_1\).
    \item The lines \(C_1,\{C_i \mid i\geq 3\}, Q_{12},\{Q_{2j} \mid j\geq 3\}\) have intersection \(2\) with \(E_1+L_{12}\).
    \item The lines \(\{L_{1j} \mid j\geq 3\}, \{Q_{ij} \mid i\geq 3\}, C_2\) have intersection \(1\) with \(E_1\) and intersection \(0\) with \(L_{12}\). The lines \(\{Q_{1j} \mid j\geq 3\}, \{L_{ij} \mid i\geq 3\}, E_2\) have intersection \(0\) with \(E_1\) and intersection \(1\) with \(L_{12}\). Together, these \(32\) lines give sections of the conic bundle \(\pi'_1\colon W_{\mathbb C}\to\mathbb P^1_{\mathbb C}\).
\end{enumerate}
Since \(W\) is the branch locus of \(\piDoubleCover\colon Y\to\mathbb P^1\times\mathbb P^2\), the preimage in \(Y\) of any of the 32 lines giving sections of \(\pi_1' \colon W \to \mathbb P^1\) is a section of \(\pi_1\).
From the explicit description, we see that each line in \eqref{item:0-intersection} is paired with a line in  \eqref{item:2-intersection}, so these sets contain the same number of lines defined over \(\mathbb R\). Moreover, since the members of \eqref{item:0-intersection} are components of singular fibers of the conic bundle \(\pi'_1\), this number must be even.

\end{proof}

\subsection{Splitting of real bitangents in the del Pezzo surface}\label{sec:UDelta} In this section we show that when \(\Delta(\mathbb R)\) consists of four ovals, each real bitangent of \(\Delta\) splits into two real lines on \(W\).  Furthermore, when \(\Delta(\mathbb R)\) is three ovals and \(Y\) has connected real locus, we show that either all of the real bitangents of \(\Delta\) split into two real lines on \(W\), or \(\Deltatilde(\mathbb R) \neq \emptyset\).  Combining this with Proposition~\ref{prop:W-lines} and Proposition~\ref{prop:Yproperties}\eqref{item:Deltatilde-point-section} yields the rationality construction.

Recall from Notation~\ref{notation:ovals} that the real components of the even degree plane curve \(\Delta\) are all ovals, and the complement of \(\Delta(\mathbb R)\) in \(\mathbb P^2(\mathbb R)\) is a the disjoint union of a non-orientable set \(U_\Delta\) and a finite number of discs. If none of the ovals of \(\Delta\) are nested, then the set \(U_\Delta\) is either equal to \((Q_1Q_3-Q_2^2<0)_{\mathbb R}\) or \((Q_1Q_3-Q_2^2>0)_{\mathbb R}\). Since \(\Delta\) is a smooth quartic, the only case for which nesting occurs is two nested ovals. In this case, either \((Q_1Q_3-Q_2^2<0)_{\mathbb R}\) or \((Q_1Q_3-Q_2^2>0)_{\mathbb R}\) is disconnected, and the disconnected set is the disjoint union of \(U_\Delta\) and a disc.

\begin{lem}[{\cite{ComessattiFondamentiPL}, \cite[Proof of Theorem 6.3]{Koll'ar-real-surfaces}}]\label{lem:outside-bitangent-split}
The real bitangents of \(\Delta\) are split in the del Pezzo surface \(W\) defined in \eqref{eqn:dP2} if and only if \(U_\Delta \subset (Q_1Q_3-Q_2^2<0)_{\mathbb R}\).
\end{lem}
\begin{proof}[Proof, c.f. {\cite[Proof of Theorem 6.3]{Koll'ar-real-surfaces}}]
A neighborhood of any line in \(\mathbb P^2\) is not orientable, so away from the points of tangency, each real bitangent of \(\Delta\) is contained in \(U_\Delta\). The preimage of a real bitangent is split in \(W\) if and only if it has a smooth real point. For a point \(P\in\mathbb P^2(\mathbb R)\setminus(Q_1=0)_{\mathbb R}\), we see from equation~\eqref{eq:localW} that the preimage of \(P\) under \(\pi_2'\) splits as two real points if and only if \((Q_1Q_3-Q_2^2)(P)\leq 0\).
\end{proof}

Before showing the results for three and four ovals, we note the following consequence of Lemma~\ref{lem:outside-bitangent-split}, which we will use in the proof of Proposition~\ref{prop:negative-outside-IJT}.

\begin{cor}\label{cor:no-branch-pts-section}
If \(U_\Delta \subset (Q_1Q_3-Q_2^2<0)_{\mathbb R}\) and \(\Gamma\) has no real Weierstrass points, then the quadric surface bundle \(\pi_1\) has a real section.
\end{cor}
\begin{proof}
The branch locus of \(\Gamma\to\mathbb P^1\) is the discriminant locus of the conic bundle \(\pi'_1\colon W\to\mathbb P^1\), so the assumption that \(\Gamma\) has no real Weierstrass points implies no component of a singular fiber of \(\pi'_1\) is defined over \(\mathbb R\).
\(W\) contains real lines by Lemma~\ref{lem:outside-bitangent-split}, so by Proposition~\ref{prop:W-lines} they all give sections of \(\pi'_1\) and hence \(\pi_1\).
\end{proof}

\begin{lem}\label{lem:4ovals-outside}
If the real isotopy class of \(\Delta\) is four ovals, then \(U_\Delta=(Q_1Q_3-Q_2^2<0)_{\mathbb R}\).
\end{lem}

\begin{proof}
Since \(\Delta(\mathbb R)\) is four ovals, the locus \((Q_1Q_3-Q_2^2\leq 0)_{\mathbb R}\) is either connected or consists of four connected components. We will show that \((Q_1Q_3-Q_2^2\leq 0)_{\mathbb R}\) is connected, and hence is equal to \(U_\Delta\).  The key input is Lemma~\ref{lem:Yleq3components}, which guarantees that the images of \(Y_{\Deltatilde/\Delta}(\mathbb R)\) and \(Y_{\Deltabar/\Delta}(\mathbb R)\) in \(\mathbb P^2(\mathbb R)\) can have at most three connected components.

First assume that \(Q_1\) is negative definite. By Lemma~\ref{lem:Impi2} the image of \(Y_{\Deltatilde/\Delta}(\mathbb R)\) under \(\pi_2(\mathbb R)\) is the locus \((Q_1Q_3-Q_2^2\leq 0)_{\mathbb R}\); hence, using Lemma~\ref{lem:Yleq3components}, \((Q_1Q_3-Q_2^2\leq 0)_{\mathbb R}\) cannot have four connected components and so must be connected. 
The same argument using the twisted double cover \(Y_{\Deltabar/\Delta}\) shows that if \(Q_1\) is positive definite, then \((Q_1Q_3-Q_2^2\leq 0)_{\mathbb R}\) is connected.
So we may reduce to the case that \(Q_1\) is indefinite.

Suppose for contradiction that \((Q_1Q_3-Q_2^2\leq 0)_{\mathbb R}\) is not connected, and hence consists of four connected components. Since \((Q_1=0)_{\mathbb R}\) is contained in the closed disc defined by one of these connected components, then either the image 
\( (Q_1 \geq 0)_{\mathbb R} \cup (Q_1Q_3-Q_2^2 \leq 0)_{\mathbb R}\)
of \(Y_{\Deltatilde/\Delta}(\mathbb R)\), or the image
\( (-Q_1 \geq 0)_{\mathbb R} \cup (Q_1Q_3-Q_2^2 \leq 0)_{\mathbb R}\)
of \(Y_{\Deltabar/\Delta}(\mathbb R)\)
has four connected components. This contradicts Lemma~\ref{lem:Yleq3components}.
\end{proof}

\begin{cor}\label{cor:4ovals_real_lines}
If \(\Delta\) has real isotopy class four ovals, then all 56 lines on \(W_{\mathbb C}\) are defined over \(\mathbb R\). If \(\Delta\) has real isotopy class three ovals and if \(U_\Delta=(Q_1Q_3-Q_2^2<0)_{\mathbb R}\), then 32 of the lines on \(W_{\mathbb C}\) are defined over \(\mathbb R\); in particular, there is a section of \(\pi_1\) defined over \(\mathbb R\).
\end{cor}

\begin{proof}
This follows from Lemmas~\ref{lem:outside-bitangent-split} and \ref{lem:4ovals-outside}, since in the four ovals case all \(28\) bitangents of \(\Delta\) are real, and in the three ovals case \(16\) of the bitangents are real.
\end{proof}

\begin{rem}
Corollary~\ref{cor:4ovals_real_lines} follows from results in the literature after identifying the del Pezzo surface \(W\) with one of the two possible real double covers of \(\mathbb P^2\) branched over the real quartic \(\Delta\).  Namely, if \(\Delta\) is defined by the real equation \(f(u,v,w) = 0\) with \((f < 0)_{\mathbb R} = U_\Delta\), then the two possible double covers are \(F_\Delta^+ = (t^2=f(u,v,w))\subset\mathbb P^3(1,1,1,2)\) and \(F_\Delta^- = (t^2=-f(u,v,w))\subset\mathbb P^3(1,1,1,2)\), as in \cite{Koll'ar-real-surfaces}.  By \cite{ComessattiFondamentiPL}, the preimages of the real bitangents of \(\Delta\) split in \(F_\Delta^-\) (see also \cite[Section 6]{Koll'ar-real-surfaces}).
\end{rem}

\begin{proof}[Proof of Proposition~\ref{prop:4ovals_rational}]
First suppose \(\Delta(\mathbb R)\) is four ovals. Then the conic bundle \(\pi_1'\colon W\to\mathbb P^1\), and hence the quadric surface bundle \(\pi_1\colon Y\to\mathbb P^1\), has a real section by Proposition~\ref{prop:W-lines} and Corollary~\ref{cor:4ovals_real_lines}. In particular \(Y\) is rational over \(\mathbb R\) by Proposition~\ref{prop:Yproperties}\eqref{item:Yquadricsurfacebundle}.

Now assume \(\Delta(\mathbb R)\) is three ovals and \(Y(\mathbb R)\) is connected. If \(U_\Delta=(Q_1Q_3-Q_2^2<0)_{\mathbb R}\), then Proposition~\ref{prop:W-lines} and Corollary~\ref{cor:4ovals_real_lines} imply that the conic bundle \(\pi_1'\colon W\to\mathbb P^1\), and hence the quadric surface bundle \(\pi_1\colon Y\to\mathbb P^1\), has a real section. So we may suppose that \(U_\Delta\) is \((Q_1Q_3-Q_2^2>0)_{\mathbb R}\), which implies that \((Q_1Q_3-Q_2^2\leq 0)_{\mathbb R}\) is disconnected. By Lemma~\ref{lem:Yleq3components}, the locus \((Q_1\geq 0)_{\mathbb R}\cup(Q_1Q_3-Q_2^2\leq 0)_{\mathbb R}\) is connected. Hence \((Q_1< 0)_{\mathbb R}\) must be contained in one of the discs of \(\mathbb P^2(\mathbb R)\setminus\Delta(\mathbb R)\). In particular \((Q_1>0)_{\mathbb R}\cap(Q_1Q_3-Q_2^2=0)_{\mathbb R}\neq\emptyset\), so \(\Deltatilde(\mathbb R)\neq\emptyset\).  By Proposition~\ref{prop:Yproperties}\eqref{item:Deltatilde-point-section} \(\pi_1 \colon Y \to \mathbb P^1\) has a real section and hence \(Y\) is rational.
\end{proof}

\subsection{Sufficiency of the intermediate Jacobian torsor obstruction when \(Q_1Q_3-Q_2^2<0\) outside \(\Delta\)}\label{sec:IJT-sufficient}
In this section, we will use the del Pezzo surface considered in the preceding sections to prove that if \(Q_1Q_3-Q_2^2\) is negative outside of the ovals of \(\Delta\), then the IJT obstruction characterizes rationality. The main result is Proposition~\ref{prop:negative-outside-IJT}, and we will obtain Theorem~\ref{thm:DoubleCoverRationalityCriteria}\eqref{item:0ovals-criterion}--\eqref{item:2nonnested-criterion} as corollaries.

Recall from Section~\ref{sec:IJT-background} that \cite{FJSVV} showed that the threefolds \(Y\) defined in Section~\ref{sec:DoubleCover} have four intermediate Jacobian torsors \(P\cong\bPic^0_\Gamma\), \(\Ptilde\), \(\Pm{1}\cong\bPic^1_\Gamma\), and \(\Ptildem{1}\), where \(\Gamma\) is the genus two curve associated to \(\Deltatilde\to\Delta\) in Proposition~\ref{prop:Yproperties}; and that \(\Ptilde+\Pm{1}=\Ptildem{1}\). Moreover, Lichtenbaum showed that \(\bPic^1_\Gamma(\bb R)\neq\emptyset\). Thus, the vanishing of the IJT obstruction is equivalent to the triviality of \(\Ptildem{1}\). In this case, \cite{FJSVV} proved that the existence of a point on the intermediate Jacobian torsor \(\Ptildem{1}\) yields a Galois-invariant geometric section of the quadric surface bundle \(\pi_1\). In general, this need not descend to a real section, since \(\Br\mathbb R\) is nontrivial. However, we will show that in the case \(Q_1Q_3-Q_2^2<0\) outside \(\Delta\) (which happens precisely when \(W\) contains real lines), we do in fact obtain a real section.

As in the previous sections \(U_\Delta\) denotes the outside of \(\Delta\) (Notation~\ref{notation:ovals}). Recall that by Lemma~\ref{lem:outside-bitangent-split}, the condition that \(U_\Delta\subset (Q_1Q_3-Q_2^2<0)_{\mathbb R}\) is equivalent to splitting of the real bitangents on the degree two del Pezzo surface \(W\) defined in \eqref{eqn:dP2}, which is the branch locus of \(Y\to\mathbb P^1\times\mathbb P^2\).

\begin{prop}\label{prop:negative-outside-IJT}
Assume \(U_\Delta\subset (Q_1Q_3-Q_2^2<0)_{\mathbb R}\). Then the quadric surface bundle \(\pi_1\colon Y\to\mathbb P^1\) has a real section if and only if the IJT obstruction vanishes for \(Y\).
\end{prop}

\begin{proof}
If \(\pi_1\) has a section, then \(Y\) is rational so the IJT obstruction vanishes by \cite[Theorem 3.11]{bw-ij}, so it remains to show the reverse implication. 
If \(\Ptildem{1}(\mathbb R)\neq\emptyset\) then by \cite[Proposition 4.5]{FJSVV} there exists a real line \(\ell\subset\mathbb P^2\) and a Galois-invariant geometric section \(\frak S\) of \(Y_\ell := Y\times_{\mathbb P^2} \ell \to\ell\) that maps with odd degree to \(\mathbb P^1\) under \(\pi_1\). \cite[Lemma 5.1]{FJSVV} and the assumption that \(U_\Delta=(Q_1Q_3-Q_2^2<0)_{\mathbb R}\) implies that every real point on \(\ell\) has preimage one or two real points in the del Pezzo surface \(W\), so \(Y_\ell\to\ell\) is surjective on real points and in the proof of \cite[Proposition 4.5]{FJSVV} \(\frak S\) may in fact be chosen to be defined over \(\mathbb R\) by Theorem~\ref{thm:witt}. By Springer's theorem \cite[Corollary 18.5]{EKM}, the quadric surface bundle \(\pi_1\) has a real section.
\end{proof}

\begin{rem}
The condition that \(U_\Delta\subset(Q_1Q_3-Q_2^2<0)_{\mathbb R}\) alone is not sufficient to guarantee rationality. The irrational example of \cite[Theorem 1.3(2)]{FJSVV} is one oval and has \(U_\Delta=(Q_1Q_3-Q_2^2<0)_{\mathbb R}\), but it has an IJT obstruction; we will generalize their example in Example~\ref{exmp:1ovalIJT} below. Example~\ref{exmp:irr2nested} will give a two nested ovals example with \(U_\Delta\subset(Q_1Q_3-Q_2^2<0)_{\mathbb R}\) (moreover in this case \(Y(\mathbb R)\) is disconnected). In these examples, the eight real lines on \(W\) are all contained in sets \eqref{item:0-intersection} and \eqref{item:2-intersection} of Proposition~\ref{prop:W-lines}.
\end{rem}

\begin{rem}\label{rem:FJSVVexmp1.5}
\cite[Example 1.5]{FJSVV} has \(\Delta(\mathbb R)\) one oval, \(U_\Delta \not\subset (Q_1Q_3-Q_2^2<0)_{\mathbb R}\), \(Y(\mathbb R)\neq\emptyset\) connected, and no IJT obstruction. Rationality of \(Y\) is still unknown in this example. We have found many additional similar examples by searching for examples where \(\Gamma\) has two real branch points and \(\pi_1(\mathbb R)\) is not surjective, and by using the code \texttt{P1tilde-bitangents.sage} in \cite{JJ-code} to verify \(\Ptildem{1}(\mathbb R)\neq\emptyset\), for example:
\[Q_1 := -u^2-v^2+w^2, \quad
Q_2 := u^2-3v^2-w^2, \quad
Q_3 := -10u^2-10v^2-w^2.\]




\end{rem}

We now apply Proposition~\ref{prop:negative-outside-IJT} to the cases of no ovals, two nested ovals, and two non-nested ovals.

\begin{cor}[Theorem~\ref{thm:DoubleCoverRationalityCriteria}\eqref{item:0ovals-criterion}]\label{cor:no-ovals-IJT}
If \(\Delta(\mathbb R)=\emptyset\) and \(Y(\mathbb R)\neq\emptyset\), then the IJT obstruction vanishes if and only if \(\pi_1\) has a section.
\end{cor}
\begin{proof}
It suffices to show that if \(\pi_1\) is not surjective on real points, then \(Y\) has an IJT obstruction to rationality.
First, note that in the \(\Delta(\mathbb R)=\emptyset\) case we have that \((Q_1Q_3-Q_2^2<0)_{\mathbb R}\) is either empty or equal to \(U_\Delta\). The assumptions that \(\Delta(\mathbb R)=\emptyset\) and \(Y(\mathbb R)\neq\emptyset\) imply that \(U_\Delta = (Q_1Q_3-Q_2^2<0)_{\mathbb R}\) and that the image of \(\pi_2(\mathbb R)\) is \(\mathbb P^2(\mathbb R)\): this is immediate from Lemma~\ref{lem:Impi2} if \(Q_1\) is negative definite, and follows from the containment \((Q_1=0)_{\mathbb R}\subset (Q_1Q_3 - Q_2^2\leq 0)_{\mathbb R}\) if \(Q_1\) is indefinite. If \(Q_1\) is positive definite and \((Q_1Q_3-Q_2^2<0)_{\mathbb R}\) is empty, then using Lemma~\ref{lem:Impi2} and the \(\PGL_2\) action on the quadratic forms \(Q_i\) described in \cite[Theorem 2.6(1)]{FJSVV}, it follows that every fiber of \(\pi_1\) has signature \((3,1)\), which contradicts the assumption that \(\pi_1(\mathbb R)\) is not surjective. Thus, we must have \(U_\Delta=(Q_1Q_3-Q_2^2<0)_{\mathbb R}=\mathbb P^2(\mathbb R)\). The claim then follows from Proposition~\ref{prop:negative-outside-IJT}.
\end{proof}

\begin{cor}\label{cor:two-nested-IJT}
Assume \(\Delta(\mathbb R)\) is two nested ovals. If the IJT obstruction vanishes, then \(Y(\mathbb R)\) is connected.
\end{cor}

\begin{proof}
Assume \(Y(\mathbb R)\) is disconnected; we will show that \(\Ptildem{1}(\mathbb R)=\emptyset\).
Disconnectedness of \(Y(\mathbb R)\) implies that \((Q_1\geq 0)_{\mathbb R}\cup (Q_1Q_3-Q_2^2\leq 0)_{\mathbb R}\) is disconnected and contains \(U_\Delta\). Since \((Q_1=0)_{\mathbb R}\subset (Q_1Q_3-Q_2^2\leq 0)_{\mathbb R}\), we must have that \(U_\Delta\subset (Q_1Q_3-Q_2^2<0)_{\mathbb R}\). If the IJT obstruction vanishes, then Proposition~\ref{prop:negative-outside-IJT} implies that the quadric surface bundle \(\pi_1\) has a section, which is impossible since \(Y(\mathbb R)\) is disconnected.
\end{proof}

\begin{cor}[Theorem~\ref{thm:DoubleCoverRationalityCriteria}\eqref{item:2nonnested-criterion}]\label{cor:two-non-nested-sufficiency}
If \(\Delta(\mathbb R)\) is two non-nested ovals and \(Y(\mathbb R)\) is connected, then \(\pi_1\) has a section if and only if the IJT obstruction vanishes.
\end{cor}

\begin{proof}
If \(U_\Delta =(Q_1Q_3-Q_2^2<0)_{\mathbb R}\) then the result follows from Proposition~\ref{prop:negative-outside-IJT}, so we may assume that \(U_\Delta\neq (Q_1Q_3-Q_2^2<0)_{\mathbb R}\). Then \((Q_1Q_3-Q_2^2<0)_{\mathbb R}\) is a disjoint union of two discs, so connectedness of \(Y(\mathbb R)\) implies \((Q_1>0)_{\mathbb R}\cap (Q_1Q_3-Q_2^2<0)_{\mathbb R}\neq\emptyset\), which implies \(\Deltatilde(\mathbb R)\neq\emptyset\).
\end{proof}

\begin{rem}\label{rem:2nonnested-criterion}
Corollary~\ref{cor:two-non-nested-sufficiency} shows that the topological criterion of connectedness of \(Y(\mathbb R)\) combined with the vanishing of the IJT obstruction is sufficient to guarantee rationality of \(Y\) in the two non-nested ovals case. In this case, neither condition alone is sufficient. Example~\ref{exmp:irr2nonnested} has \(Y(\mathbb R)\) connected but has an IJT obstruction, and in \cite[Theorem 1.3(1)]{FJSVV} the IJT obstruction vanishes but \(Y(\mathbb R)\) is disconnected.
\end{rem}

\begin{rem}
Throughout this section we have assumed that \(Q_1Q_3-Q_2^2\) is negative outside the ovals. In the case when \(Q_1Q_3-Q_2^2\) is instead negative inside the ovals, we can immediately determine rationality in several cases: if \(\Delta(\mathbb R)=\emptyset\) then exactly one of \(Y_{\Deltatilde/\Delta}\) or \(Y_{\Deltabar/\Delta}\) is rational, and the other has no real points; and if \(\Delta(\mathbb R)\) is two or three non-nested ovals, then exactly one of \(Y_{\Deltatilde/\Delta}\) or \(Y_{\Deltabar/\Delta}\) is rational, and the other has disconnected real locus.
However, in the cases of one oval or two non-nested ovals, rationality is less clear. In these cases, at least one of \(Y_{\Deltatilde/\Delta}\) or \(Y_{\Deltabar/\Delta}\) is rational by Corollary~\ref{cor:YorTwistRational}.
Example~\ref{exmp:rational2nestedovals} shows that for two nested ovals case, it is possible for both \(Y_{\Deltatilde/\Delta}\) and \(Y_{\Deltabar/\Delta}\) to be rational. For one oval, \cite[Example 1.5]{FJSVV} gives an example where \(Y_{\Deltabar/\Delta}\) is rational, but rationality of \(Y_{\Deltatilde/\Delta}\) is unknown; see also Remark~\ref{rem:1oval-rational-question}.
\end{rem}

\section{Construction of examples}\label{sec:ExamplesConstruction}

In this section, we construct examples of conic bundles by giving equations for quadrics \(Q_1,Q_2,Q_3\) and taking \(Y:= Y_{\Deltatilde/\Delta}\) and \(\Deltatilde\to\Delta\) to be as defined in Section~\ref{sec:DoubleCover}.
Our examples are constructed in the same manner as those of Frei--Ji--Sankar--Viray--Vogt.
The topological type of \(\Delta(\mathbb R)\) is determined using the \texttt{Sage} code accompanying \cite{PSV-quartics}. Smoothness of \(\Delta\) and \(\Deltatilde\), and the numerical claims about the signatures of the fibers of \(\pi_1\) can be verified by hand or with the code \texttt{Quadric-bundle-verifications.sage} in our \texttt{GitHub} respository \cite{JJ-code}, which is a \texttt{Sage} implementation of the \texttt{Magma} code accompanying \cite{FJSVV}. By deforming the coefficients in each example, one can obtain similar examples of each type; we refer the interested reader to the \texttt{Macaulay2} code \texttt{Singular-members.m2} in \cite{JJ-code}, which one can use to find singular members in such a one-parameter family.

We first construct irrational examples in Section~\ref{sec:IrrationalExamples}, where irrationality is witnessed either by the IJT obstruction (Section~\ref{sec:IJT-sufficient}) or by the real locus of \(Y(\mathbb R)\). In Section~\ref{sec:IJT-failure}, we show that several of these examples are irrational despite having no IJT obstruction, further demonstrating the insufficiency of this obstruction when \(Q_1Q_3-Q_2^2\) is negative inside \(\Delta\). Finally, we construct rational examples in Section~\ref{sec:RationalExamples}.

Before giving our constructions, we first outline where they fit in by giving a proof of Theorem~\ref{thm:main}:
\begin{proof}[Proof of Theorem~\ref{thm:main}]
As mentioned in the introduction, \eqref{item:Deltatildepoint_rational} is \cite[Proposition 6.1]{FJSVV}. For \eqref{item:0_oval},
Example~\ref{exmp:rationalDeltaempty} is rational, Example~\ref{exmp:0ovalIJT} is irrational and has real points, and Example~\ref{exmp:DeltaemptyIrr} has no real points. For \eqref{item:1_oval}, Example~\ref{exmp:rational1oval} is rational and \cite[Theorem 1.3(2)]{FJSVV} is irrational (Example~\ref{exmp:1ovalIJT} gives additional examples with the same obstruction).
For~\eqref{item:2_nonnested}, Example~\ref{exmp:rational2nonnestedovals}
is rational, Example~\ref{exmp:irr2nonnested} is irrational and connected, and \cite[Theorem 1.3(1)]{FJSVV} is irrational and disconnected.
For~\eqref{item:disconnected-types}, in the two nested ovals case Examples~\ref{exmp:rational2nestedovals}\eqref{item:2nested-Deltatildeempty}--\eqref{item:2nested-Deltatildepoint} are rational and Example~\ref{exmp:irr2nested} is irrational, and in the three ovals case Example~\ref{exmp:rational3ovals} is rational and Example~\ref{exmp:irr3oval} is irrational. Finally, Example~\ref{exmp:rational4ovals} shows \eqref{item:4_ovals}.
\end{proof}

\subsection{Construction of irrational examples}\label{sec:IrrationalExamples}
We will now construct the irrational examples of Theorem~\ref{thm:main}.
Example~\ref{exmp:irr2nonnested} below, together with \cite[Theorem 1.3(1)]{FJSVV}, shows the necessity of both the topological and IJT conditions in Theorem~\ref{thm:DoubleCoverRationalityCriteria}\eqref{item:2nonnested-criterion}.

In our setting, many obstructions to rationality automatically vanish:
\(\Br Y\cong\Br\mathbb R\), the intermediate Jacobian of \(Y\) is isomorphic to \(\bPic^0_\Gamma\), and the unramified cohomology groups are trivial whenever \(Y(\mathbb R)\) is connected; see \cite[Section 1.1]{FJSVV} for details. Our examples will use the topological and IJT obstructions:
Example~\ref{exmp:DeltaemptyIrr} has a real points obstruction to (uni)rationality;
Examples~\ref{exmp:0ovalIJT}, \ref{exmp:1ovalIJT}, and \ref{exmp:irr2nonnested} have an IJT obstruction to rationality; and Examples~\ref{exmp:irr2nested} and \ref{exmp:irr3oval} have a real components obstruction to (stable) rationality.

\begin{exmp}[Pointless example with \(\Delta(\mathbb R)=\emptyset\)]\label{exmp:DeltaemptyIrr}
Let \(Y\) be the double cover of \(\mathbb P^1\times\mathbb P^2\) constructed in Section~\ref{sec:DoubleCover} for the quadrics
\[ Q_1 := -u^2-v^2-w^2, \quad
Q_2 := -u^2-v^2+w^2, \quad
Q_3 := -2u^2 -9v^2 - 3w^2. \]
Then \(\Delta(\mathbb R)=\emptyset\), \(\Gamma\) is defined by \(y^2=t^6 + 2t^5 + 10t^4 + 4t^3 + 19t^2 + 30t + 54\), and \(\Gamma\) has no real Weierstrass points. In particular, \(\Gamma(\mathbb R)\neq\emptyset\) is connected. The fibers of \(\pi_1\) all have signature \((0,4)\), so \(Y(\mathbb R)=\emptyset\).
\end{exmp}

The next three examples are unirational, and irrationality is witnessed by the IJT obstruction.

\begin{exmp}[Irrational example with \(\Delta(\mathbb R)=\emptyset\) and \(Y(\mathbb R)\neq\emptyset\)]\label{exmp:0ovalIJT}
By Corollary~\ref{cor:no-ovals-IJT}, if \(Q_1\) is positive definite and \(Q_3\) is negative definite, and the resulting \(\Deltatilde\to\Delta\) is an \'etale cover of a smooth curve, then both \(Y_{\Deltatilde/\Delta}\) and \(Y_{\Deltabar/\Delta}\) have IJT obstructions to rationality. For an explicit example, one may take \[Q_1 := u^2+v^2+w^2, \quad Q_2 := u^2 - v^2, \quad Q_3 := -u^2-v^2-9w^2.\]
\end{exmp}

\begin{exmp}[Irrational example with \(\Delta(\mathbb R)\) one oval]\label{exmp:1ovalIJT}
Let \(Q_1,Q_2,Q_3\) be quadrics such that the resulting \(\Deltatilde\to\Delta\) is an \'etale cover of a smooth curve and such that \((Q_1Q_3-Q_2^2\leq 0)_{\mathbb R}\) is a non-orientable subset of \(\mathbb P^2(\mathbb R)\). If \(\pi_1\) is not surjective on real points, then \(Y_{\Deltatilde/\Delta}\) has an IJT obstruction by Proposition~\ref{prop:negative-outside-IJT}.

\cite[Theorem 1.3(2)]{FJSVV} gives an explicit example of such a choice of quadrics. Alternatively, one may also take the following (noting that \((u=0)\subset(Q_1Q_3-Q_2^2 < 0)_{\mathbb R}\), so \((Q_1Q_3-Q_2^2 < 0)_{\mathbb R}\) is not orientable): \[Q_1 := -u^2-v^2+w^2, \quad Q_2 := u^2 - uv + 3v^2, \quad Q_3 := -u^2+v^2+2vw-10w^2.\]
\end{exmp}

\begin{figure}[h]
\caption{The regions \((Q_1Q_3-Q_2^2\leq 0)_{\mathbb R}\) in blue and \((Q_1\geq 0)_{\mathbb R}\) in red in Example~\ref{exmp:irr2nonnested} (left), Example~\ref{exmp:irr2nested} (center) and Example~\ref{exmp:irr3oval} (right), on the affine open chart \((w\neq 0)\). In each example, the image of \(Y(\mathbb R)\) in \(\mathbb P^2(\mathbb R)\) is equal to \((Q_1Q_3-Q_2^2\leq 0)_{\mathbb R}\).}
\includegraphics[height=50mm]{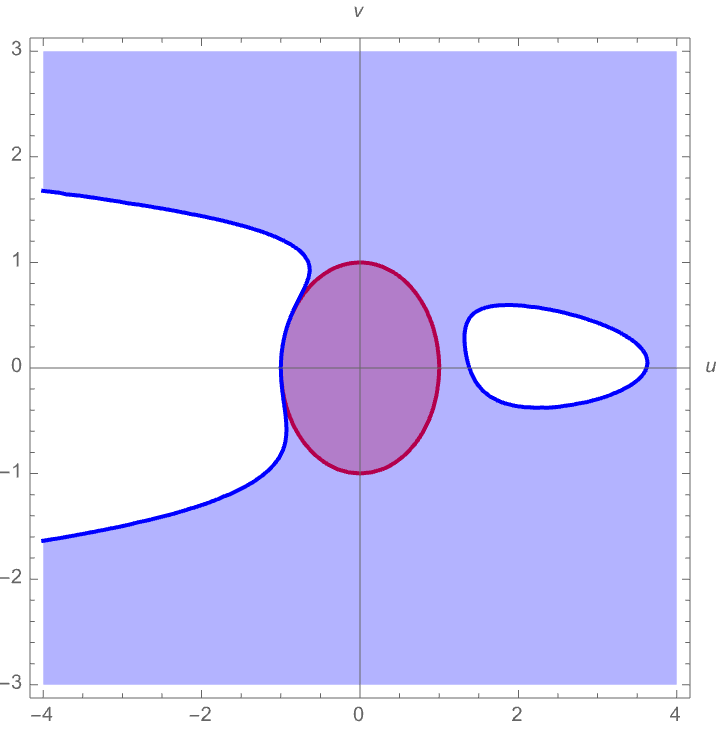} \quad
\includegraphics[height=50mm]{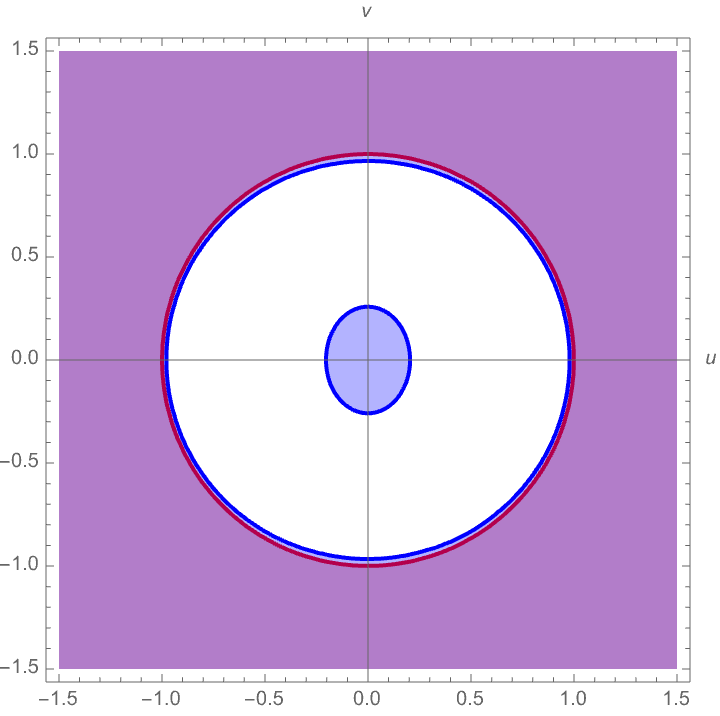} \quad
\includegraphics[height=50mm]{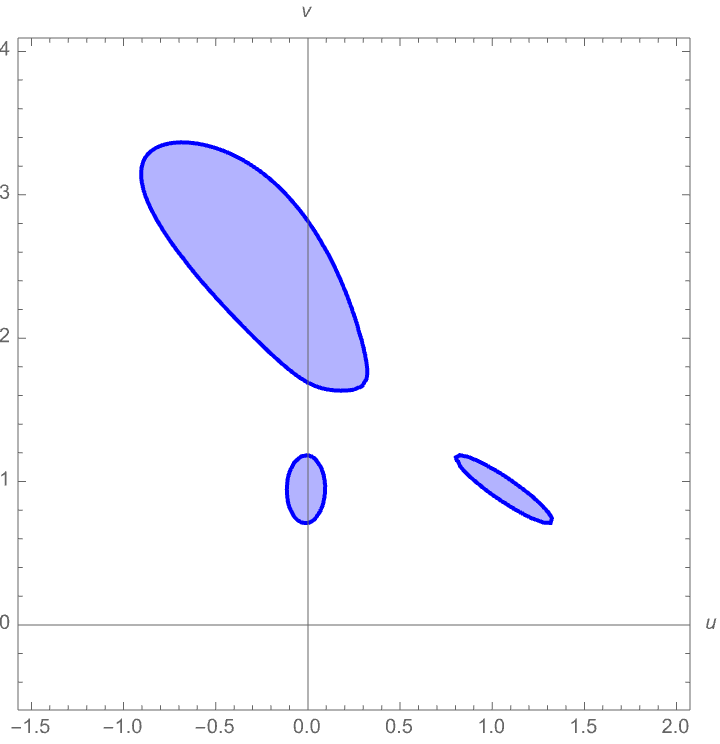}
\label{fig:3oval-IJT}
\end{figure}

\begin{exmp}[Irrational, connected example with \(\Delta(\mathbb R)\) two non-nested ovals]\label{exmp:irr2nonnested}

Define
\[Q_1 := -u^2-v^2+w^2, \quad
Q_2 := u^2+3v^2+uw-vw, \quad
Q_3 := -u^2+v^2+2vw-10w^2\]
and let \(Y\) be the associated double cover of \(\mathbb P^1\times\mathbb P^2\) constructed in Section~\ref{sec:DoubleCover}. Then \(\Delta(\mathbb R)\) is two non-nested ovals. \(\Gamma\) is defined by \(y^2=-t^6 + 8t^5 - 4t^4 - 66t^3 + 116t^2 - 36t - 11\) and has six real branch points \([t:1]\) with \(t\approx -2.9708, -0.1845, 0.7545, 1.5708, 2.8152, 6.0149\). The signatures of the fibers \(Y_{[t : 1]}\) are as follows:
\begin{center}
\begin{tabular}{ c | c | c | c | c | c | c}
 \(t\) & \(-2\) & \(0\) & \(1\) & \(2\) & \(4\) & \(7\) \\ 
 \hline
 Signature & \((0,4)\) & \((1,3)\) & \((2,2)\) & \((1,3)\) & \((2,2)\) & \((1,3)\)
\end{tabular}
\end{center}
In particular, \(Y(\mathbb R)\) is connected and \(\pi_1\) is not surjective on real points. This implies \(\Deltatilde(\mathbb R)=\emptyset\) by Proposition~\ref{prop:Yproperties}\eqref{item:Deltatilde-point-section}, so \(U_\Delta=(Q_1Q_3-Q_2^2<0)_{\mathbb R}\). Thus, by Proposition~\ref{prop:negative-outside-IJT}, \(Y\) has an IJT obstruction to rationality.

\end{exmp}

In the following examples, we construct conic bundles with disconnected real loci. These examples are unirational but not stably rational.
Recall from Corollary~\ref{cor:components-of-Y-and-Delta} that if \(Y\) is constructed as in Section~\ref{sec:DoubleCover} and has disconnected real locus, then \(\Delta(\mathbb R)\) must be two or three ovals. The following examples, together with \cite[Theorem 1.3(1)]{FJSVV} show these cases all occur:
Frei--Ji--Sankar--Viray--Vogt gave an example where \(\Delta(\mathbb R)\) is two non-nested ovals and \(Y(\mathbb R)\) has two connected components, and here we use their methods to give examples with two nested ovals and three ovals.

\begin{exmp}[Disconnected example with \(\Delta(\mathbb R)\) two nested ovals]\label{exmp:irr2nested}
Define
\[
Q_1 :=  u^2 + v^2 - w^2, \quad
Q_2 :=  u^2 + v^2, \quad
Q_3 :=  -24u^2 - 15v^2 + w^2,
\]
and let \(Y\) be the associated double cover of \(\mathbb P^1\times\mathbb P^2\) constructed in Section~\ref{sec:DoubleCover}.
Then \(\Delta(\mathbb R)\) is two nested ovals, and \(\Gamma\) is defined by \(y^2 = t^6 + 4t^5 - 36t^4 - 82t^3 + 395t^2 + 78t - 360\) and has real Weierstrass points over \([t:1]\) with \(t = -6, -5, -1, 1, 3, 4\). The signatures of the fibers \(Y_{[t : 1]}\) are as follows.
\begin{center}
\begin{tabular}{ c | c | c | c | c | c | c}
 \(t\) & \(-5.5\) & \(-3\) & \(0\) & \(2\) & \(3.5\) & \(5\) \\ 
 \hline
 Signature & \((1,3)\) & \((0,4)\) & \((1,3)\) & \((0,4)\) & \((1,3)\) & \((2,2)\)
\end{tabular}
\end{center}
Thus, \(Y(\mathbb R)\) has two connected components by Lemma~\ref{lem:ImComponents}.
\end{exmp}

\begin{exmp}[Disconnected example with \(\Delta(\mathbb R)\) three ovals]
\label{exmp:irr3oval}
Let \(Y\) be as defined in Section~\ref{sec:DoubleCover} for
\begin{gather*} Q_1 := -2u^2 - 2uv + 4uw - 2v^2 + 6vw - 5w^2, \quad
Q_2 := 10uv - 20uw + 5v^2 - 20vw + 20w^2, \\
Q_3 := -48u^2 - 48uv + 96uw - 20v^2 + 88vw - 92w^2. \end{gather*}
Then \(\Delta(\mathbb R)\) is three ovals. The genus two curve \(\Gamma\) is defined by \(y^2 = t^6 - 56t^4 + 784t^2 - 2304\) and has six real Weierstrass points over \(t = -6, -4, -2, 2, 4, 6\). The signatures of the \(Y_{[t : 1]}\) are as follows.
 \begin{center}
\begin{tabular}{ c | c | c | c | c | c | c}
 \(t\) & \(-5\) & \(-3\) & \(0\) & \(3\) & \(5\) & \(7\) \\ 
  \hline
 Signature & \((1,3)\) & \((0,4)\) & \((1,3)\) & \((0,4)\) & \((1,3)\) & \((0,4)\)
\end{tabular}
\end{center}
Thus, \(Y(\mathbb R)\) has three connected components by Lemma~\ref{lem:ImComponents}.
\end{exmp}

\begin{rem}
If \(X\) is a smooth complete intersection of two quadrics in \(\mathbb P^5\) that contains a conic \(C\) defined over \(\mathbb R\), then projection from the conic realizes the blow up of \(X\) along \(C\) as a conic bundle with quartic discriminant curve \cite[Remark 13]{ht-intersection-quadrics}. Krasnov's topological classification of intersections of quadrics \cite[Theorem 5.4]{Krasnov-biquadrics} shows that a conic bundle arising in such a way can have at most two real connected components, so in particular the conic bundle of Example~\ref{exmp:irr3oval} is not birational over \(\mathbb R\) to an intersection of two quadrics.
In the following section, we will also see using \cite[Corollary 6.3]{FJSVV} that Example~\ref{exmp:DeltaemptyIrr} cannot be obtained from an intersection of two quadrics by projection from a conic.
\end{rem}

\begin{rem}\label{rem:signatures}
When \(\pi_1\) is not surjective on real points, the signature sequence of the fibers of \(\pi_1\) appears to dictate the properties of \(Y\) and \(\Delta\). The examples we have found have all followed the following pattern:
\begin{enumerate}
\item
Signatures \((0,4),(1,3)\): Experimentally, these examples behave like \cite[Example 1.5]{FJSVV}: \(\Delta(\mathbb R)\) is one oval, \(Q_1Q_3-Q_2^2>0\) outside \(\Delta\), and \(\Ptildem{1}(\mathbb R)\neq\emptyset\). 
In this case we cannot determine rationality.
\item
Signatures \((0,4),(1,3),(2,2),(1,3)\): Experimentally, these examples behave like Example~\ref{exmp:1ovalIJT}: \(\Delta(\mathbb R)\) is one oval and \(Q_1Q_3-Q_2^2<0\) outside \(\Delta\), so there is an IJT obstruction to rationality.
\item
Signatures \((0,4),(1,3),(2,2),(1,3),(2,2),(1,3)\): Experimentally the examples have exhibited the behavior of Example~\ref{exmp:irr2nonnested}: \(\Delta(\mathbb R)\) is two non-nested ovals, and there is an IJT obstruction.
\item
Signatures \((0,4),(1,3),(2,2),(3,1),(2,2),(1,3)\): This is the setting of Corollary~\ref{cor:no-ovals-IJT}. In this case \(\Delta(\mathbb R)\) is empty, the image of \(\pi_2(\mathbb R)\) is \(\mathbb P^2(\mathbb R)\), and there is an IJT obstruction.
\item
Signatures \((0,4),(1,3),(0,4),(1,3)\): \(Y(\mathbb R)\) has two connected components, and experimentally \(\Delta(\mathbb R)\) has been two non-nested ovals and \(\Ptildem{1}(\mathbb R)\neq\emptyset\).
\item
Signatures \((0,4),(1,3),(0,4),(1,3),(2,2),(1,3)\): \(Y(\mathbb R)\) has two connected components and \(\Delta(\mathbb R)\) must be two nested ovals by Remark~\ref{rem:2,2sig-nested}. \(Y\) has both a topological and IJT obstruction by Corollary~\ref{cor:two-nested-IJT}.
\item
Signatures \((0,4),(1,3),(0,4),(1,3),(0,4),(1, 3)\): \(Y(\mathbb R)\) has three components, and \(\Delta(\mathbb R)\) is necessarily three ovals by Corollary~\ref{cor:components-of-Y-and-Delta}. Experimentally \(\Ptildem{1}\) has \(\mathbb R\)-points.
\end{enumerate}
We do not know if there exists a two nested ovals example with \(Y(\mathbb R)\) connected and \(\pi_1\) not surjective on real points. If one could use the above signature sequences to give an isotopy classification for \(Y\), in the style of the intersection of two quadrics situation \cite{AgrachevGamkrelidze,Krasnov-biquadrics}, one might be able to show that in the two nested ovals case \(Y\) is rational \(\iff\) \(Y(\mathbb R)\) is connected \(\iff\) the IJT obstruction vanishes.
\end{rem}

\subsection{Failure of the intermediate Jacobian torsor obstruction}\label{sec:IJT-failure}

Proposition~\ref{prop:negative-outside-IJT} shows that the IJT obstruction characterizes rationality for \(Y\) if \(Q_1Q_3-Q_2^2\) is negative on the outside of \(\Delta\) (which implies that \(Y(\mathbb R)\neq\emptyset\)).
However, the IJT obstruction is no longer sufficient to show rationality when if \(Q_1Q_3-Q_2^2\) is negative inside the ovals of \(\Delta\), as shown by the two non-nested ovals example of \cite[Theorem 1.3(1)]{FJSVV}. In this section, we use the techniques of \cite[Theorem 1.3(1)]{FJSVV} to show that Examples~\ref{exmp:DeltaemptyIrr} and \ref{exmp:irr3oval} give further examples of this failure of the IJT obstruction to characterize rationality.

Recall that \(\Delta\) and \( (Q_1Q_3-Q_2^2<0)_{\mathbb R}\) each have no real points in Example~\ref{exmp:DeltaemptyIrr}, and that \(Y(\mathbb R)\) and \(\Delta(\mathbb R)\) each have three connected components in Example~\ref{exmp:irr3oval}. So \(Q_1Q_3-Q_2^2\) is negative inside \(\Delta\) in both of these examples. Before proving that the IJT obstruction vanishes in these examples, we first review the strategy of \cite{FJSVV}. Recall from Section~\ref{sec:IJT-background} that \cite[Section 2]{FJSVV} gives an explicit description of points on the intermediate Jacobian torsor \(\Ptildem{1}\). Applying their criterion, Frei--Ji--Sankar--Viray--Vogt show that IJT obstruction also fails to characterize rationality for geometrically standard conic bundles over \(\mathbb P^2\) by constructing an example of a conic bundle whose real locus has two connected components and such that there is a Galois-invariant set of four points of \(\Deltatilde\) spanning a \(3\)-plane in \(\mathbb P^4\) and whose pushforward under \(\piDelta\) is \(\Delta\cap(w=0)\). Thus, they explicitly exhibit a real point on the intermediate Jacobian torsor \(\Ptildem{1}\).
We apply their methods to Examples~\ref{exmp:DeltaemptyIrr} and \ref{exmp:irr3oval}:

\begin{prop}\label{prop:IJTvanishirrational}
All the intermediate Jacobian torsors are trivial in Examples~\ref{exmp:DeltaemptyIrr} and \ref{exmp:irr3oval}. In particular, these conic bundles have no IJT obstruction to rationality, and the real locus of \(Y\) exhibits irrationality.
\end{prop}

We will use the line \((w=0)\) in Example~\ref{exmp:DeltaemptyIrr}, and the line \((v=0)\) in Example~\ref{exmp:irr3oval}. In the argument for Example~\ref{exmp:irr3oval} below, we exhibit a point on \(\Ptildem{1}\) using the \texttt{Sage} code \texttt{P1tilde.sage} in \cite{JJ-code}; this code checks the rank of the matrix obtained from a Galois-invariant set of four points on \(\Deltatilde(\mathbb C)\) mapping to \(\Delta\cap\ell\) for any line \(\ell\) defined over \(\mathbb Q\) and that meets \(\Delta\) in four distinct complex points. In our \texttt{GitHub} repository \cite{JJ-code}, we also include the code \texttt{P1tilde-bitangents.sage}, which does the analogous check when \(\ell\) is a real bitangent of \(\Delta\); the code that computes the real bitangents of \(\Delta\) is due to Plaumann--Sturmfels--Vinzant and is included in the supplementary material for their paper \cite{PSV-quartics}.

\begin{proof}[Proof for Example~\ref{exmp:DeltaemptyIrr} (No ovals, \(Y(\mathbb R)=\emptyset\))] The quartic curve \(\Delta\) is defined by the equation \(u^4 + 9u^2v^2 + 7u^2w^2 + 8v^4 + 14v^2w^2 + 2w^4=0\), and the intersection \(\Delta\cap(w=0)\) consists of the four complex points
\([-i:1:0], [i:1:0], [-2i\sqrt{2}:1:0],[2i\sqrt{2}:1:0]\).
One verifies that the set
\[[i:1:0:0:i\sqrt{7}], \quad [-i:1:0:0:-i\sqrt{7}], \quad
[2i\sqrt{2}:1:0:\sqrt{7}:\sqrt{7}], \quad [-2i\sqrt{2}:1:0:\sqrt{7}:\sqrt{7}]\]
of four points of \(\Deltatilde\) is \(\Gal(\mathbb C/\mathbb R)\)-invariant and maps to \(\Delta\cap (w=0)\). Since
\[\det\begin{pmatrix}
i & 1 & 0 & i\sqrt{7} \\
-i & 1 & 0 & -i\sqrt{7} \\
2i\sqrt{2} & 1 & \sqrt{7} & \sqrt{7} \\
-2i\sqrt{2} & 1 & \sqrt{7} & \sqrt{7}
\end{pmatrix} = -56\sqrt{2}\neq 0,\]
the four points above span a \(3\)-plane in \(\mathbb P^4\), so \(\Ptildem{1}(\mathbb R)\neq\emptyset\).
\end{proof}

\begin{proof}[Proof for Example~\ref{exmp:irr3oval} (Three ovals, \(Y(\mathbb R)\) three connected components)]
The quartic curve \(\Delta\) is defined by the equation \(96u^4 + 192u^3v + 132u^2v^2 + 36uv^3 + 15v^4 - 384u^3w - 448u^2vw - 136uv^2w - 96v^3w + 408u^2w^2 + 152uvw^2 + 212v^2w^2 - 48uw^3 - 192vw^3 + 60w^4\), and its restriction to the line \((v=0)\) is given by \(96 u^4 - 384 u^3 w + 408 u^2 w^2 - 48 u w^3 + 60 w^4=0\), so \(\Delta\cap(v=0)\) consists of four points with approximate coordinates
\[[-0.01183 \pm 0.38575 i:0:1], \quad [2.0118 \pm 0.38575 i:0:1].\]
One can verify by hand or using \texttt{P1tilde.sage} in the accompanying code \cite{JJ-code} that the set
\[[-0.01183 \pm 0.38575 i:0:1 : -0.35355 \mp 2.20794i : 1.97589 \pm 9.48178 i],\]
\[[2.0118 \pm 0.38575 i:0:1 : 0.35355 \mp 2.20794 i :  1.97589 \mp 9.48178 i]\]
is a \(\Gal(\mathbb C/\mathbb R)\)-invariant set of points on \(\Deltatilde(\mathbb C)\) mapping to \(\Delta\cap(v=0)\), and they span a 3-plane in \(\mathbb P^4\) since
\[\det \begin{pmatrix}-0.01183 + 0.38575 i & 1 & -0.35355 - 2.20794i & 1.97589 + 9.48178 i \\
-0.01183 - 0.38575 i & 1 & -0.35355 + 2.20794 i & 1.97589 - 9.48178 i \\
2.0118 + 0.38575 i & 1 & 0.35355 - 2.20794 i &  1.97589 - 9.48178 i \\
2.0118 - 0.38575 i & 1 & 0.35355 + 2.20794 i & 1.97589 + 9.48178 i\end{pmatrix} \approx
-359.61663 \neq 0.\]

In this example, \(Y\) does not contain any real points above the line \((v=0)\), as illustrated in Figure~\ref{fig:3oval-IJT}. Thus, one can see concretely that the Galois-invariant geometric section constructed from the point on \(\Ptildem{1}\) by \cite[Proposition 4.5]{FJSVV} does not descend to a real section.

\end{proof}
	
\subsection{Construction of rational examples}\label{sec:RationalExamples}

In this section, we construct rational examples of double covers as in Section~\ref{sec:DoubleCover}.
We will focus on two cases: (1) \(\Deltatilde\) has no real points, and (2) \(Q_1Q_3-Q_2^2<0\) inside \(\Delta\).

Recall that if \(\Deltatilde\to\Delta\) is any \'etale double cover of a smooth plane quartic, then Theorem~\ref{thm:bruin} (\cite{bruin}) shows that it can be realized in the form Equation~\eqref{eqn:DeltatildeDelta}. If \(\Deltatilde\) has a point, then the corresponding double cover is automatically rational by Proposition~\ref{prop:Yproperties}\eqref{item:Deltatilde-point-section}. Thus, the question is more interesting when \(\Deltatilde(\mathbb R)=\emptyset\).

We also consider examples where \(Q_1Q_3-Q_2^2<0\) inside \(\Delta\) (Notation~\ref{notation:ovals}) because Section~\ref{sec:IJT-sufficient} already gives a rationality criterion in the opposite case. (Recall this never happens for four ovals by Lemma~\ref{lem:4ovals-outside}.) At least one of \(Y_{\Deltatilde/\Delta}\) or the twisted double cover \(Y_{\Deltabar/\Delta}\) is rational. When \(\Delta\) is two non-nested ovals or three ovals, then exactly one of the two is rational and the other has disconnected real locus; the twisted double cover corresponding to \cite[Theorem 1.3(1)]{FJSVV} and Example~\ref{exmp:irr3oval} give such examples.

When \(\Delta\) is one oval and \(Q_i\) are quadratic forms such that \(Y_{\Deltatilde/\Delta}\) has the properties in \cite[Example 1.5]{FJSVV} (see Remark~\ref{rem:FJSVVexmp1.5}), then the twisted cover produces a rational example with \(\Deltabar(\mathbb R)\neq\emptyset\) and \(Q_1Q_3-Q_2^2<0\) inside \(\Delta\). However, we are not able to construct a one oval example with \(\Deltabar(\mathbb R)=\emptyset\), see Remark~\ref{rem:1oval-rational-question}. In the two nested ovals case, we construct examples both with and without points on \(\Deltatilde\). In particular Example~\ref{exmp:rational2nestedovals}\eqref{item:2nested-Deltatildeempty} exhibits an example where \(\pi_1\colon Y\to\mathbb P^1\) has a section that does not come from any known rationality construction: \(\Deltatilde(\mathbb R)=\emptyset\), and moreover the conic bundle \(Y_\ell\to\ell\) over any real line \(\ell\subset\mathbb P^2\) has no real sections, so the rationality construction does not come from Section~\ref{sec:dP2}.

Throughout, we will compute the signatures of the fibers of the quadric surface bundle \(\pi_1\) to show that it is surjective on real points. Theorem~\ref{thm:witt} (\cite{Witt-quadratic-forms}) then implies that \(\pi_1\) has a real section; thus, \(Y\) is rational. We also note that in the following examples, with the exception of Example~\ref{exmp:rationalDeltaempty}, both \(Y_{\Deltatilde/\Delta}\) and the twisted double cover \(Y_{\Deltabar/\Delta}\) are rational.

\begin{exmp}[Rational example with \(\Delta(\mathbb R)=\emptyset\)]\label{exmp:rationalDeltaempty}
Let \(Q_1,Q_2,Q_3\) be as in Example~\ref{exmp:DeltaemptyIrr},
and let \(Y := Y_{\Deltabar/\Delta}\) be the twisted double cover defined in Definition~\ref{defn:TwistedDoubleCover}.
Then \(\Delta(\mathbb R)=\emptyset\), and \(\Gamma\) is defined by \(y^2=-t^6 + 2t^5 - 10t^4 + 4t^3 - 19t^2 + 30t - 54\). We note that \(\Gamma(\mathbb R)=\emptyset\), so in particular \(\Gamma\) has no real Weierstrass points. Every fiber of \(\pi_1\) has signature \((3,1)\), so \(\pi_1\) has a section.

In this example \((Q_1Q_3-Q_2^2\leq 0)_{\mathbb R}=\emptyset\), and the image of \(Y(\mathbb R)\to\mathbb P^2(\mathbb R)\) is \((-Q_1\geq 0)_{\mathbb R}=\mathbb P^2(\mathbb R)\).
\end{exmp}

\begin{figure}[h]
\caption{The image of \(Y(\mathbb R)\) on the chart \((w\neq 0)\) in Example~\ref{exmp:rational1oval} (left), Example~\ref{exmp:rational2nestedovals}\eqref{item:2nested-Deltatildeempty} (center), and Example~\ref{exmp:rational2nestedovals}\eqref{item:2nested-Deltatildepoint} (right). The region \((Q_1Q_3-Q_2^2\leq 0)_{\mathbb R}\) is in blue, \((Q_1\geq 0)_{\mathbb R}\) is in red, and the real bitangents of \(\Delta\) are shown in black.}
\includegraphics[height=50mm]{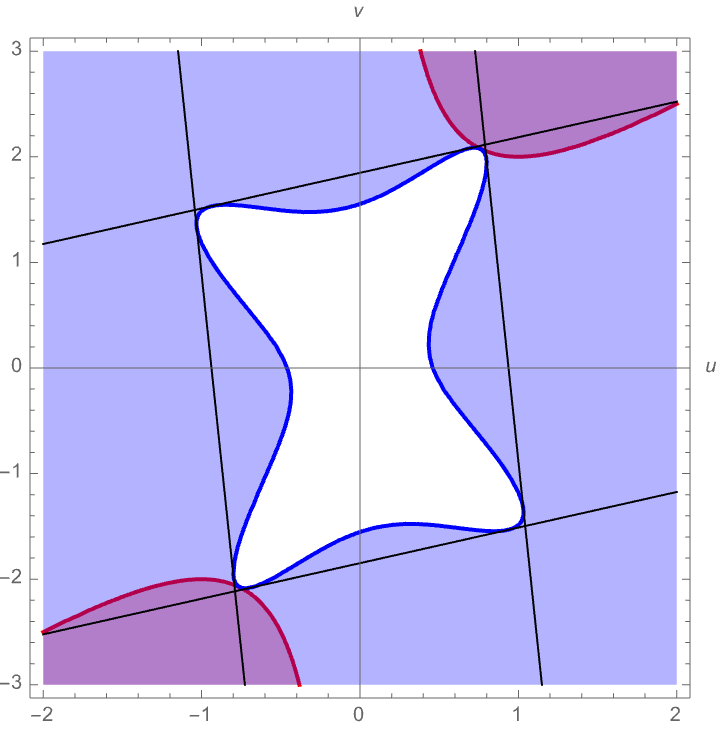}\quad
\includegraphics[height=50mm]{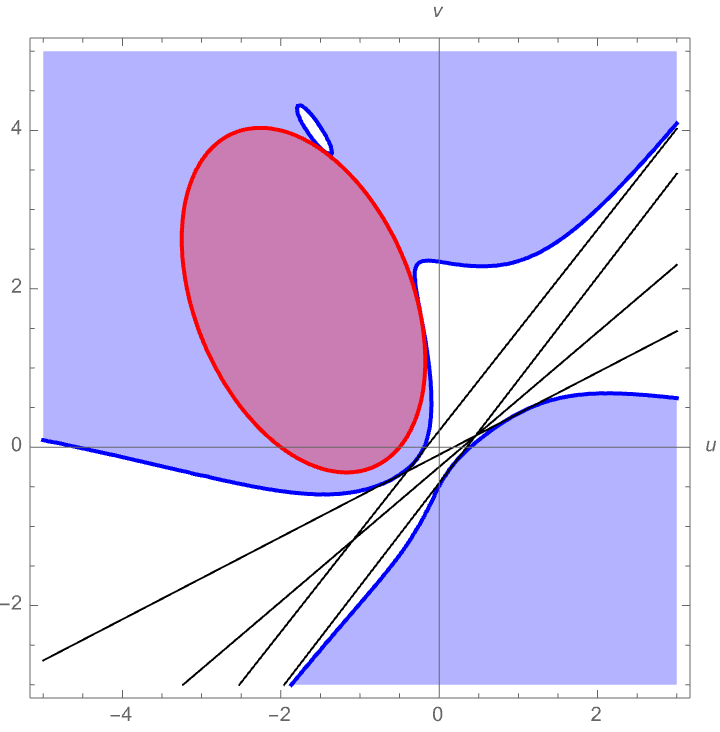}\quad
\includegraphics[height=50mm]{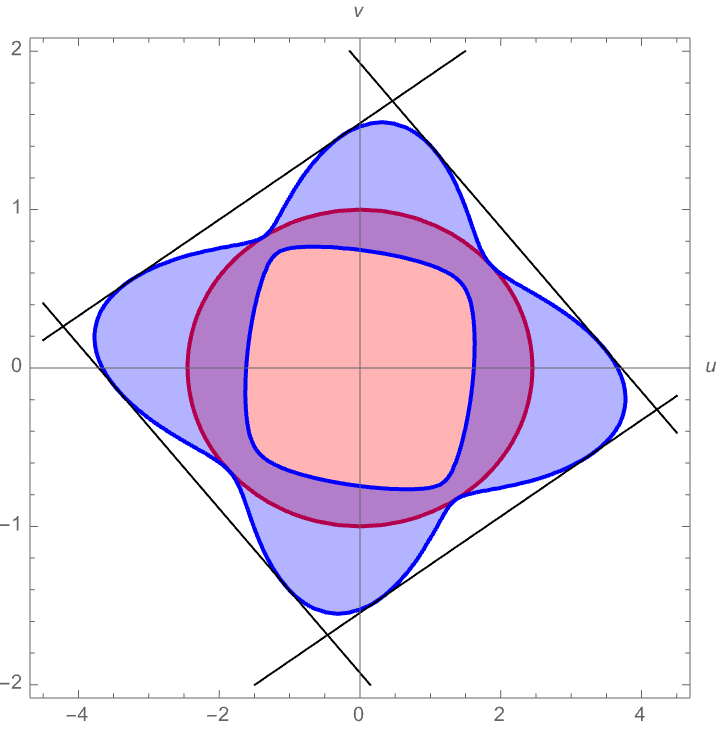}
\label{fig:rational-examples-1}
\end{figure}

\begin{exmp}[Rational example with \(\Delta(\mathbb R)\) one oval, \(\Deltatilde(\mathbb R)=\emptyset\), \(Q_1Q_3 - Q_2^2<0\) outside \(\Delta\)]\label{exmp:rational1oval}
Define
\[     Q_1 := - u^2  + uv - w^2, \quad
    Q_2 := 3u^2 + uv - v^2 +  w^2, \quad
    Q_3 := - u^2 - 2uv - 2w^2, \quad \]
and let \(\Deltatilde\to\Delta\) and \(Y_{\Deltatilde/\Delta}\) be as defined in Section~\ref{sec:DoubleCover}. Then \(\Delta(\mathbb R)\) is one oval, and we claim \(\Deltatilde(\mathbb R)\) is empty.

For this, since \(\Delta(\mathbb R)\) is connected and the zero locus \((Q_1 = 0)_{\mathbb R}\) is contained in \((Q_1Q_3-Q_2^2 \leq 0)_{\mathbb R}\), so it suffices to check that \(\Delta\) has an \(\mathbb R\)-point \(P\) such that \(Q_1(P) < 0\). Indeed, one verifies the restriction of \(Q_1\) to the line \((v=0)\) is negative definite, and \(\Delta(\mathbb R)\cap(v=0)_{\mathbb R}\neq\emptyset\). Next, the genus two curve \(\Gamma\) is defined by \(y^2=-\tfrac{1}{4}t^6 + \tfrac{3}{2}t^5 - \tfrac{29}{2}t^4 + 30t^3 - 33t^2 + 10t - 2\) and has no real points. Every fiber of \(\pi_1\) has signature \((1, 3)\), and so \(\pi_1(\mathbb R)\) is surjective and \(Y\) is rational.

The line \((w=0)\) is contained in \((Q_1Q_3-Q_2^2<0)_{\mathbb R}\) since the restriction of \(Q_1Q_3-Q_2^2\) to \((w=0)\) is the equation \(-8 u^4 - 5 u^3 v + 3 u^2 v^2 + 2 u v^3 - v^4\), which is always negative. Thus \((Q_1Q_3-Q_2^2<0)_{\mathbb R}\) cannot be orientable, so \(U_\Delta=(Q_1Q_3-Q_2^2<0)_{\mathbb R}\). This example is depicted in Figure~\ref{fig:rational-examples-2}.


\end{exmp}

\begin{exmp}[Rational examples with \(\Delta(\mathbb R)\) two nested ovals, \(Q_1Q_3-Q_2^2<0\) inside \(\Delta\)]\label{exmp:rational2nestedovals} $ $
\begin{enumerate}
\item\label{item:2nested-Deltatildeempty}
(\(\Deltatilde(\mathbb R)\) is empty.)
Let \(\Deltatilde\to\Delta\) and \(Y_{\Deltatilde/\Delta}\) be as defined in Section~\ref{sec:DoubleCover} for the quadrics
\begin{gather*}Q_1 := -4u^2 - 2uv - 2v^2 - 10uw + 4vw - 4w^2, \quad Q_2 := u^2 - 4uv - 3v^2 - 6uw + 2vw + 2w^2, \\
Q_3 := -u^2 - 6uv + 8uw - 6vw - 3w^2.\end{gather*}
Then \(\Delta(\mathbb R)\) is two nested ovals, and we claim that \(\Deltatilde(\mathbb R)\) is empty (see Figure~\ref{fig:rational-examples-1}).

To show \(\Deltatilde(\mathbb R)=\emptyset\), on the chart \((w \neq 0)\) we define the box \(B := \{(u, v)\ |\ -2 \leq u \leq -1,\ 3.5 \leq v \leq 4.5\}\). One checks that the boundary of \(B\) is disjoint from \((\Delta=0)_{\mathbb R}\), that the set \((\Delta=0)_{\mathbb R} \cap (Q_1 < 0)_{\mathbb R} \cap B\) contains an \(\mathbb R\)-point \((u, v)\) with \(v = 4\) and \(-2 < u < -1\), and the set \((\Delta=0)_{\mathbb R} \cap (Q_1 < 0)_{\mathbb R}\) contains an \(\mathbb R\)-point in the complement of \(B\) whose \(v\)-coordinate is \(1\).
In particular, there are points on both connected components of \((\Delta=0)_{\mathbb R}\) where \(Q_1\) is negative. Since \((Q_1 \leq 0)_{\mathbb R}\) is connected and \((Q_1 = 0)_{\mathbb R} \subset (Q_1Q_3-Q_2^2 \leq 0)_{\mathbb R}\), we have that \(Q_1 \leq 0\) on \((\Delta=0)_{\mathbb R}\).

Next, genus two curve \(\Gamma\) is defined by \(y^2=-58t^6 - 398t^5 - 677t^4 + 244t^3 + 394t^2 - 24t - 108\) and has no real points. All fibers of \(\pi_1\) have signature \((1, 3)\), so \(\pi_1(\mathbb R)\) is surjective.

In this example, \(Y\setminus Y_\Delta\) does not contain any real points above a real bitangent of \(\Delta\); in fact \(Y_\ell(\mathbb R)\to\ell(\mathbb R)\) is not surjective for any real line \(\ell\subset\mathbb P^2\). Thus, the section of \(\pi_1\) is not obtained from a curve lying over \(Y_\ell\) as in Section~\ref{sec:dP2}, and it is not constructed from a point on \(\Deltatilde\) as in Proposition~\ref{prop:Yproperties}\eqref{item:Deltatilde-point-section} since \(\Deltatilde\) has no points.



\item\label{item:2nested-Deltatildepoint}
(\(\Deltatilde\) has an \(\mathbb R\)-point.)
Let \(\Deltatilde\to\Delta\) and \(Y_{\Deltatilde/\Delta}\) be as defined in Section~\ref{sec:DoubleCover} for the quadrics
\[    Q_1 := -u^2  - 6v^2 + 6w^2, \quad
    Q_2 := -u^2 +uv+ 3v^2 + w^2, \quad
    Q_3 := -2u^2 - 6v^2 + 6w^2.\]
Then \(\Delta(\mathbb R)\) is two nested ovals, the genus two curve \(\Gamma_{\Deltatilde/\Delta}\) is defined by the equation \(y^2=-36t^6 - 48t^5 - 78t^4 - 46t^3 - 102t^2 - 24t - 72\) and has no real branch points, and every fiber of \(Y_{\Deltatilde/\Delta}\to\mathbb P^1\) has signature \((1,3)\). Figure~\ref{fig:rational-examples-1} gives a visual depiction of this example. In this case, none of the real bitangents of \(\Delta\) give sections of \(\pi_1\) as in Section~\ref{sec:dP2}; however, \(\Deltatilde\) has \(\mathbb R\)-points, which can be used to construct sections of \(\pi_1\).

In this case, the image of \(\piDelta(\mathbb R)\) is one oval of \(\Delta(\mathbb R)\). To show this, we consider the points \(P_0 := [0:0:1]\in (Q_1>0)_{\mathbb R}, P_1:= [0:1:1]\in (Q_1>0)_{\mathbb R},\) and \(P_2:= [0:2:1]\in (Q_1>0)_{\mathbb R}\). Since \((Q_1Q_3-Q_2^2)(P_0)=35>0\), \((Q_1Q_3-Q_2^2)(P_1)=-16<0\), and \((Q_1Q_3-Q_2^2)(P_2)=115>0\), and since we know that \((Q_1=0)_{\mathbb R}\) is contained in \((Q_1Q_3-Q_2^2\leq 0)_{\mathbb R}\), we conclude that \(Q_1\) is nonnegative on exactly one of the two ovals, and is nonpositive on the other. Thus, \(\Deltatilde\) and \(\Deltabar\) both have real points, so both \(Y_{\Deltatilde/\Delta}\) and \(Y_{\Deltabar/\Delta}\) are rational.
The twisted double cover \(Y_{\Deltabar/\Delta}\) has the property that \(\Gamma_{\Deltabar/\Delta}\) has no Weierstrass points, and the fibers of \(Y_{\Deltabar/\Delta}\to\mathbb P^1\) all have signature \((2,2)\).
\end{enumerate}
\end{exmp}

\begin{rem}
One can check using \texttt{P1tilde-bitangents.sage} in \cite{JJ-code} that in the preceding Examples~\ref{exmp:rationalDeltaempty}, \ref{exmp:rational1oval}, and \ref{exmp:rational2nestedovals}, \(\Ptildem{1}\) contains points mapping to all of the real bitangents. This gives rationality constructions for Examples~\ref{exmp:rationalDeltaempty} and \ref{exmp:rational1oval} by the proof of Proposition~\ref{prop:negative-outside-IJT}, but not for the ``points inside \(\Delta\)" conic bundles of Example~\ref{exmp:rational2nestedovals}.
\end{rem}

\begin{exmp}[Rational example with \(\Delta(\mathbb R)\) two non-nested ovals, \(\Deltatilde(\mathbb R)=\emptyset\)]\label{exmp:rational2nonnestedovals}
Define
\[     Q_1 := -u^2 + uv + v^2 + vw, \quad
    Q_2 := -2uv + vw + w^2, \quad
    Q_3 := u^2 - v^2 - 2uw, \]
and let \(\Deltatilde\to\Delta\) and \(Y_{\Deltatilde/\Delta}\) be as defined in Section~\ref{sec:DoubleCover}.
Then \(\Delta(\mathbb R)\) is two non-nested ovals.

To show that \(\Deltatilde(\mathbb R)=\emptyset\), we work on the chart \((w \neq 0)\).
One can verify that the lines \((w=0)\) and \((2v = -1)\) are disjoint from \((\Delta=0)_{\mathbb R}\); that the sets \(\Delta(\mathbb R)\cap(v=0)\) and \(\Delta(\mathbb R)\cap(v=-1)\) are both nonempty, and that \(Q_1\) is nonpositive on each of these two sets. Since \(\Delta(\mathbb R)\) has two connected components, we conclude that \(Q_1\) is nonpositive on \((\Delta=0)_{\mathbb R}\). See Figure~\ref{fig:rational-examples-2} for a visual depiction.

The genus two curve \(\Gamma\) is defined by \(y^2 = -\tfrac{1}{4}t^6 + \tfrac{3}{2}t^5 - \tfrac{17}{4} t^4 + 4t^3 - 2t^2 + 2t - 1\) and has two real branch points, and the signatures of the fibers have sequence \((2,2),(1,3)\). Therefore \(\pi_1(\mathbb R)\) is surjective.
\end{exmp}

\begin{figure}[h]
\caption{The regions \((Q_1Q_3-Q_2^2\leq 0)_{\mathbb R}\) in blue and \((Q_1\geq 0)_{\mathbb R}\) in red, in Example~\ref{exmp:rational2nonnestedovals} (left), Example~\ref{exmp:rational3ovals} (center), and Example~\ref{exmp:rational4ovals} (right), on the chart \((w\neq 0)\). The real bitangents of \(\Delta\) are shown in black. For each, the image of \(Y(\mathbb R)\) is \((Q_1Q_3-Q_2^2\leq 0)_{\mathbb R}\).}
\includegraphics[height=50mm]{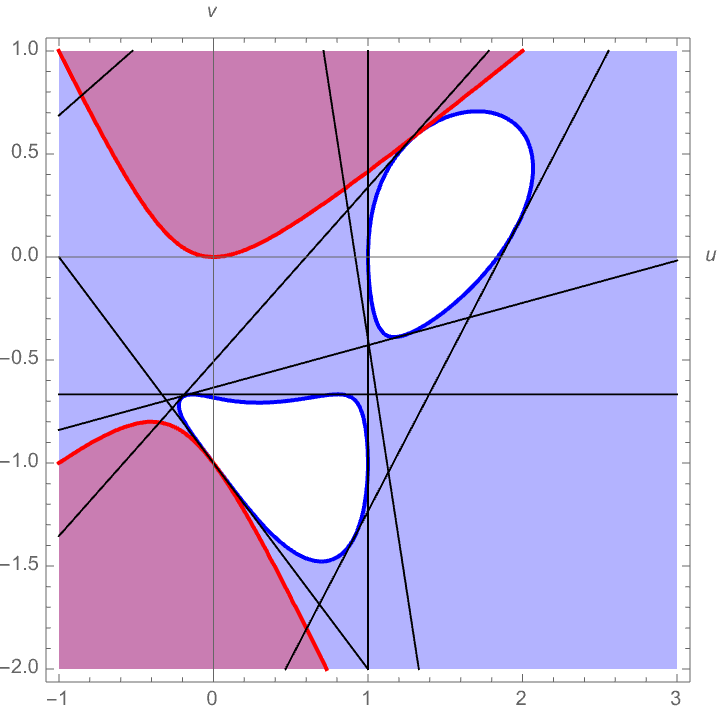}\quad
\includegraphics[height=50mm]{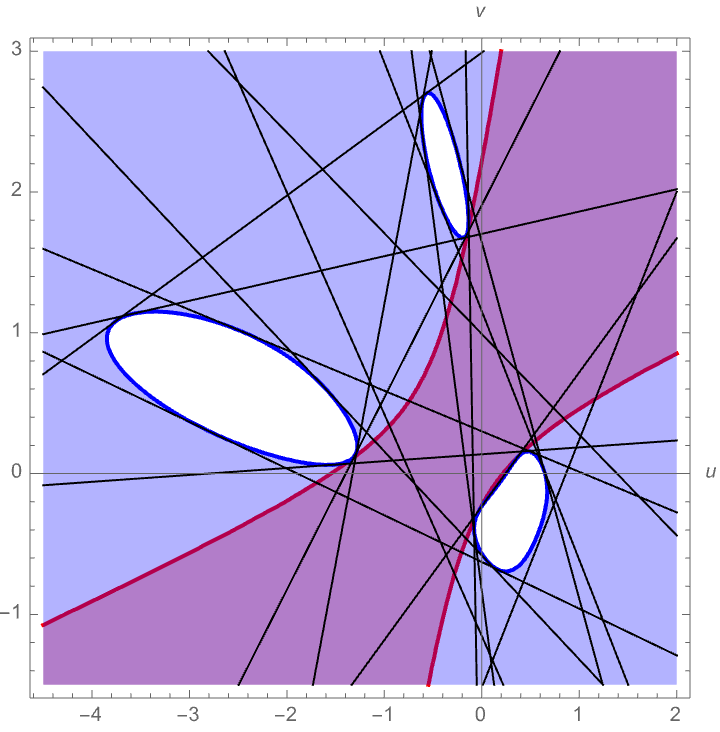}\quad
\includegraphics[height=50mm]{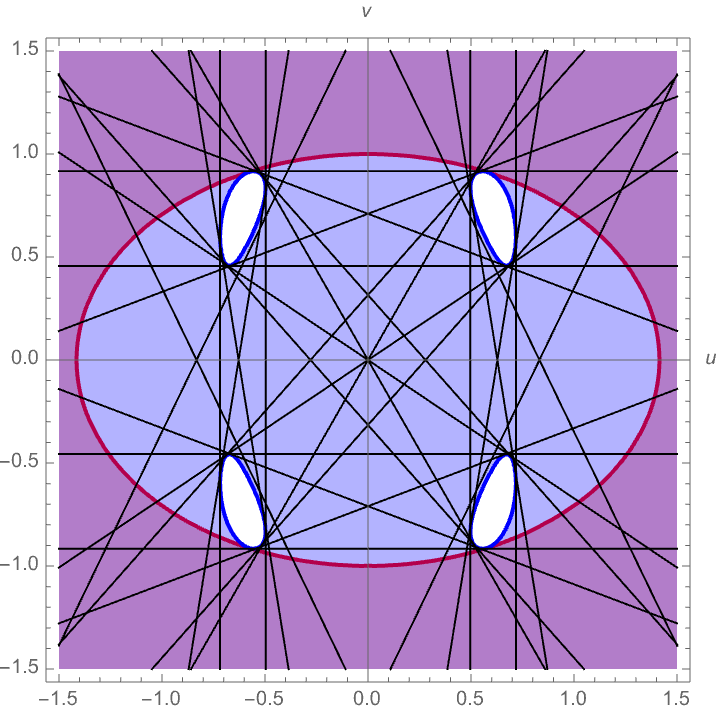}
\label{fig:rational-examples-2}
\end{figure}

\begin{exmp}[Rational example with \(\Delta(\mathbb R)\) three ovals, \(\Deltatilde(\mathbb R)=\emptyset\)]\label{exmp:rational3ovals}
Define
\begin{gather*}Q_1 := -3u^2 + 10uv - 2v^2 - 4uw + 4vw + w^2,\quad Q_2 := 5u^2 + 8uv + 5v^2 + 4uw - 6vw - 2w^2, \\ Q_3 := -2u^2 - 8uv - 2v^2 + 2uw + 2vw - 3w^2,\end{gather*}
and let \(\Deltatilde\to\Delta\) and \(Y_{\Deltatilde/\Delta}\) be as defined in Section~\ref{sec:DoubleCover}.
Then \(\Delta(\mathbb R)\) is three ovals.

To show that \(\Deltatilde(\mathbb R)=\emptyset\), one can argue as in Example~\ref{exmp:rational2nonnestedovals}, checking that on the chart \((w\neq 0)\) one oval is to the left of \((u=-1)\), the two ovals to the right of \((u=-1)\) are separated by the line \((v=1)\), and that \(Q_1 \leq 0\) on each oval.
The associated genus two curve \(\Gamma\) has equation \(y^2 = 39t^6 + 102t^5 - 1335t^4 + 1114t^3 + 47t^2 + 20t - 32\) and has four real Weierstrass points. The fibers of \(\pi_1\) have signature sequence \((1,3),(2,2),(1,3),(2,2)\), so \(\pi_1\) is surjective on real points. See Figure~\ref{fig:rational-examples-2}.
\end{exmp}

\begin{exmp}[Rational example with \(\Delta(\mathbb R)\) four ovals, \(\Deltatilde(\mathbb R)=\emptyset\)]\label{exmp:rational4ovals}
Define the quadrics
\[     Q_1 := u^2 + 2v^2 - 2w^2, \quad
    Q_2 := 3u^2 - w^2, \quad
    Q_3 := -2u^2 - v^2 + w^2,\]
and let \(\Deltatilde\to\Delta\) and \(Y_{\Deltatilde/\Delta}\) be as defined in Section~\ref{sec:DoubleCover}.
Then \(\Delta(\mathbb R)\) is four ovals.
To see that \(\Deltatilde(\mathbb R)=\emptyset\), one can work on the chart \((w \neq 0)\) and show that each quadrant of the the \((u,v)\) plane contains an oval, and then argue as in the previous examples. See Figure~\ref{fig:rational-examples-2} for a visual depiction.

Theorem~\ref{thm:DoubleCoverRationalityCriteria}\eqref{item:4ovals-rational} implies \(Y\) is rational. One can also check this explicitly: \(\Gamma\) is defined by \(y^2 = 4t^6 + 28t^5 + 12t^4 - 34t^3 - 3t^2 + 10t - 2\) and has six real branch points, and 
the fibers \(Y_{[t:1]}\) have signatures as shown below.
\begin{center}
\begin{tabular}{ c | c | c | c | c | c | c}
 \(t\) & \(-7\) & \(-3\) & \(-1\) & \(0\) & \(0.35\) & \(0.5\) \\ 
 \hline
 Signature & \((2,2)\) & \((1,3)\) & \((2,2)\) & \((1,3)\) & \((2,2)\) & \((1,3)\)
\end{tabular}
\end{center}
\end{exmp}

\begin{rem}
If \(\Delta(\mathbb R)\) is four ovals, then \(\Gamma\) has six real Weierstrass points by Proposition~\ref{prop:W-lines} and Corollary~\ref{cor:4ovals_real_lines}.
\end{rem}

\begin{rem}\label{rem:1oval-rational-question}
All the rational examples with \(\Deltatilde(\mathbb R)=\emptyset\) we constructed above have the property that \((Q_1=0)_{\mathbb R}\) intersects every component of \(\Delta(\mathbb R)\). One may wonder if a one oval example with \(Q_1Q_3-Q_2^2<0\) inside \(\Delta\) can be constructed in a similar manner. We have not been able to construct such a rational example: in every example we have found, \(\Gamma\) has two real Weierstrass points and \(\pi_1\) is not surjective on real points. \cite[Example 1.5]{FJSVV} gives an explicit example. One can also take the quadrics defined in Remark~\ref{rem:FJSVVexmp1.5}: here \((Q_1=0)_{\mathbb R}\) and \(\Delta(\mathbb R)\) have the points \([1:0:1]\) and \([-1:0:1]\) in common (see Figure~\ref{fig:1oval-unknown}), but \(Y\) has no real points over \([t:1]\in \mathbb P^1(\mathbb R)\) for \(1-\sqrt{2} < t <1+\sqrt{2}\).
\end{rem}

\begin{figure}[h]
\caption{The regions \((Q_1Q_3-Q_2^2\leq 0)_{\mathbb R}\) in blue and \((Q_1\geq 0)_{\mathbb R}\) in red for the quadrics defined in Remark~\ref{rem:FJSVVexmp1.5}, on the chart \((w\neq 0)\). Rationality is unknown in this case.}
\includegraphics[height=50mm]{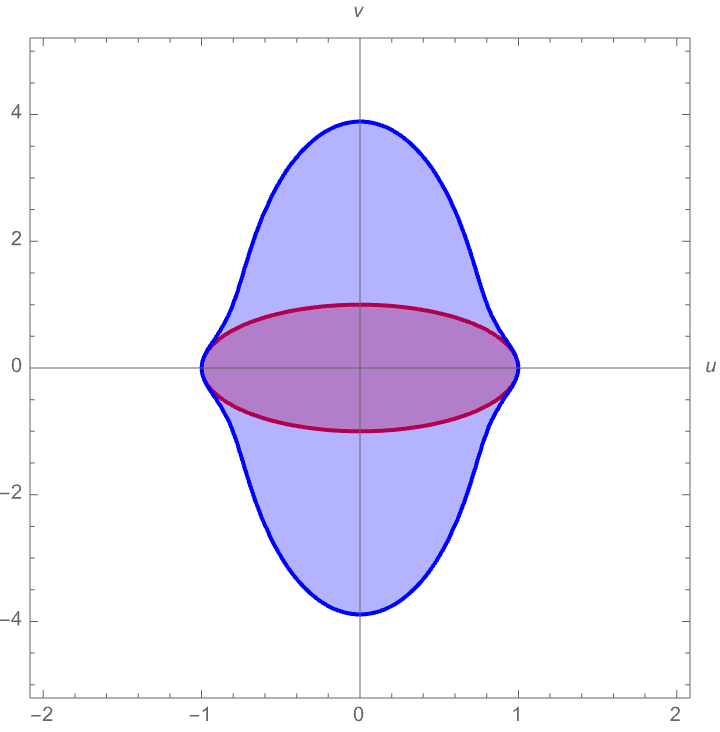}
\label{fig:1oval-unknown}
\end{figure}

\bibliographystyle{alpha}
\bibliography{references}

\end{document}